\documentclass[reqno]{amsproc}

\usepackage{graphicx}

\usepackage{hyperref}

\hypersetup{
	colorlinks=true,
	linkcolor=blue,
	anchorcolor=blue,
	citecolor=blue
}

\usepackage{amssymb}
\usepackage{amsfonts}
\usepackage{amsmath}

\newtheorem{theorem}{Theorem}[section]
\newtheorem{lemma}[theorem]{Lemma}
\newtheorem{proposition}[theorem]{Proposition}
\newtheorem{corollary}[theorem]{Corollary}

\theoremstyle{definition}

\newtheorem{remark}[theorem]{Remark}

\numberwithin{equation}{section}

\newcommand{\RNum}[1]{\uppercase\expandafter{\romannumeral #1\relax}}
\newtheorem*{defin}{Definition}
\def\Xint#1{\mathchoice
	{\XXint\displaystyle\textstyle{#1}}%
	{\XXint\textstyle\scriptstyle{#1}}%
	{\XXint\scriptstyle\scriptscriptstyle{#1}}%
	{\XXint\scriptscriptstyle\scriptscriptstyle{#1}}%
	\!\int}
\def\XXint#1#2#3{{\setbox0=\hbox{$#1{#2#3}{\int}$ }
		\vcenter{\hbox{$#2#3$ }}\kern-.6\wd0}}

\def\dashint{\Xint-}

\begin{document}

\title[Higher integrability for nonlocal equations]{Higher integrability for nonlinear nonlocal equations with irregular kernel}

\author{Simon Nowak}
\address{Universit\"at Bielefeld, Fakult\"at f\"ur Mathematik, Postfach 100131, D-33501 Bielefeld, Germany}
\email{simon.nowak@uni-bielefeld.de}
\thanks{Supported by SFB 1283 of the German Research Foundation.}



\begin{abstract}
	We prove a higher regularity result for weak solutions to nonlinear nonlocal equations along the integrability scale of Bessel potential spaces $H^{s,p}$ under a mild continuity assumption on the kernel. By embedding, this also yields regularity in Sobolev-Slobodeckij spaces $W^{s,p}$. Our approach is based on a characterization of Bessel potential spaces in terms of a certain nonlocal gradient-type operator and a perturbation approach commonly used in the context of local elliptic equations in divergence form.
\end{abstract}

\maketitle

\section{Introduction} 
\subsection{Basic setting and main result}
In this paper, we consider nonlinear nonlocal equations of the form 
\begin{equation} \label{nonlocaleq}
L_A^\Phi u = F \text{ in } \Omega \subset \mathbb{R}^n,
\end{equation}
where $s \in (0,1)$, $\Omega \subset \mathbb{R}^n$ is a domain (= open set), while $A:\mathbb{R}^n \times \mathbb{R}^n \to \mathbb{R}$ is a coefficient and $\Phi:\mathbb{R} \to \mathbb{R}$ is a nonlinearity with properties to be specified below. Moreover, the nonlocal operator $L_A^\Phi$ is formally given by
$$ L_A^\Phi u(x) = p.v. \int_{\mathbb{R}^n} \frac{A(x,y)}{|x-y|^{n+2s}} \Phi(u(x)-u(y))dy.$$
We assume that the right-hand side $F$ of (\ref{nonlocaleq}) is formally of the form
\begin{equation} \label{F}
F(x) = p.v. \int_{\mathbb{R}^n} \frac{g(x,y)}{|x-y|^{n+2s}} dy + f(x), \quad x \in \Omega,
\end{equation}
where $f:\mathbb{R}^n \to \mathbb{R}$ and $g:\mathbb{R}^n \times \mathbb{R}^n \to \mathbb{R}$ are given functions.
The aim of this work is to generalize an approach introduced in \cite{Me}, in order to prove a higher regularity result for weak solutions of the equation (\ref{nonlocaleq}) along the integrability scale of Bessel potential spaces $H^{s,p}$, in the case when the coefficient $A$ exhibits a potentially very irregular behaviour. Throughout the paper, for simplicity we assume that $n>2s$. Moreover, we assume that $A$ is a measurable function and that there exists some $\lambda \geq 1$ such that
\begin{equation} \label{eq1}
\lambda^{-1} \leq A(x,y) \leq \lambda \text{ for almost all } x,y \in \mathbb{R}^n.
\end{equation}
Furthermore, we require $A$ to be symmetric, i.e.
\begin{equation} \label{symmetry}
A(x,y)=A(y,x) \text{ for almost all } x,y \in \mathbb{R}^n.
\end{equation}
We call such a function $A$ a kernel coefficient and define $\mathcal{L}_0(\lambda)$ as the class of all such measurable kernel coefficients $A$ that satisfy (\ref{eq1}) and (\ref{symmetry}).
Moreover, in our main results $\Phi:\mathbb{R} \to \mathbb{R}$ is assumed to be a continuous function satisfying $\Phi(0)=0$ and the following Lipschitz continuity and monotonicity assumptions, namely
\begin{equation} \label{PhiLipschitz}
|\Phi(t)-\Phi(t^\prime)| \leq \lambda |t-t^\prime| \text{ for all } t,t^\prime \in \mathbb{R}
\end{equation}
and
\begin{equation} \label{PhiMonotone}
\left (\Phi(t)-\Phi(t^\prime) \right )(t-t^\prime) \geq \lambda^{-1} (t-t^\prime)^2 \text{ for all } t,t^\prime \in \mathbb{R},
\end{equation}
where for simplicity we use the same constant $\lambda \geq 1$ as in (\ref{eq1}). In particular, $\Phi$ could be any $C^1$ function with $\Phi(0)=0$ such that the first derivative $\Phi^\prime$ of $\Phi$ satisfies $\textnormal{im } \Phi^\prime \subset [\lambda^{-1},\lambda]$. In the case when $\Phi(t)=t$, the operator $L_A^\Phi$ reduces to a linear nonlocal operator widely considered in the literature. \newline
The following nonlocal analogue of the euclidean norm of the gradient of a function plays a key role in this paper.
\begin{defin}
	Let $s \in (0,1)$. For any measurable function $u:\Omega \to \mathbb{R}$, we define the s-gradient $\nabla^s u:\mathbb{R}^n \to [0,\infty]$ by 
	$$ \nabla^s u(x)= \left ( \int_{\mathbb{R}^n} \frac{(u(x)-u(y))^2}{|x-y|^{n+2s}}dy \right )^{\frac{1}{2}}.$$
\end{defin}
For any $p \in [2,\infty)$, define the space
$$H^{s,p}(\Omega | \mathbb{R}^n)= \left \{u:\mathbb{R}^n \to \mathbb{R} \text{ measurable } \mathrel{\Big|} \int_{\Omega}|u(x)|^pdx+ \int_{\Omega} |\nabla^s u(x)|^p dx < \infty \right \}.$$
Moreover, by $H^{s,p}_{loc}(\Omega|\mathbb{R}^n)$ we denote the set of all functions $u:\mathbb{R}^n \to \mathbb{R}$ that belong to $H^{s,p}(\Omega^\prime|\mathbb{R}^n)$ for any relatively compact open subset $\Omega^\prime$ of $\Omega$.
The main relevance of these spaces is due to the fact that they are closely related to the classical Bessel potential spaces $H^{s,p}(\Omega)$ and Sobolev-Slobodeckij spaces $W^{s,p}(\Omega)$. In fact, for any $p \geq 2$ we have the inclusions
\begin{equation} \label{relationsSob}
H^{s,p}(\mathbb{R}^n) \subset H^{s,p}_{loc}(\Omega|\mathbb{R}^n) \subset H^{s,p}_{loc}(\Omega) \subset W^{s,p}_{loc}(\Omega),
\end{equation}
see section 3. \newline
Denote by $H_c^{s,2}(\Omega)$ the set of all functions that belong to $H^{s,2}(\Omega|\mathbb{R}^n)$ and are compactly supported in $\Omega$.
For all measurable functions $u,\varphi:\mathbb{R}^n \to \mathbb{R}$, we define
$$\mathcal{E}_A^\Phi(u,\varphi) = \int_{\mathbb{R}^n} \int_{\mathbb{R}^n} \frac{A(x,y)}{|x-y|^{n+2s}} \Phi(u(x)-u(y))(\varphi(x)-\varphi(y))dydx,$$
provided the above expression is well-defined and finite, this is for example true if $u \in H^{s,2}_{loc}(\Omega | \mathbb{R}^n)$ and $\varphi \in H_c^{s,2}(\Omega).$
Furthermore, throughout this paper we assume that the function $g$ is measurable and symmetric in the sense of (\ref{symmetry}). In addition, by a slight abuse of notation we define the $s$-gradient $\nabla^s g:\mathbb{R}^n \to [0,\infty]$ of $g$ by 
$$ \nabla^s g(x)= \left ( \int_{\mathbb{R}^n} \frac{g(x,y)^2}{|x-y|^{n+2s}}dy \right )^{\frac{1}{2}}.$$
Also, for any such function $g$ that satisfies $\nabla^s g \in L^2_{loc}(\Omega)$ and any $\varphi \in H_c^{s,2}(\Omega)$, we define 
$$\mathcal{E}(g,\varphi) = \int_{\mathbb{R}^n} \int_{\mathbb{R}^n} \frac{g(x,y)}{|x-y|^{n+2s}}(\varphi(x)-\varphi(y))dydx.$$
The notation introduced above allows us to define our notion of weak solutions to the equation (\ref{nonlocaleq}) as follows.
\begin{defin}
	Given $f \in L^\frac{2n}{n+2s}_{loc}(\Omega)$ and a measurable symmetric function $g:\mathbb{R}^n \times \mathbb{R}^n \to \mathbb{R}$ with $\nabla^s g \in L^2_{loc}(\Omega)$, assume that $F$ is given as in (\ref{F}). We say that $u \in H^{s,2}_{loc}(\Omega | \mathbb{R}^n)$ is a weak solution of the equation $L_A^\Phi u = F$ in $\Omega$, if 
	$$ \mathcal{E}_A^\Phi(u,\varphi) = \mathcal{E}(g,\varphi) + (f,\varphi)_{L^2(\Omega)} \quad \forall \varphi \in H^{s,2}_c(\Omega).$$
\end{defin}
In our main result, we need to impose the following additional continuity assumption on $A$
\begin{equation} \label{contkernel}
\lim_{h \to 0} \sup_{x,y \in K} |A(x+h,y+h)-A(x,y)| =0 \quad \text{for any compact set } K \subset \Omega.
\end{equation}
The condition (\ref{contkernel}) was introduced in the recent paper \cite{MeH} in the context of obtaining higher H\"older regularity. In particular, it is satisfied if $A$ is either continuous in $\Omega \times \Omega$ or if $A$ is translation invariant inside of $\Omega$, that is, if there exists a measurable function $a: \mathbb{R}^n \to \mathbb{R}$ such that $A(x,y)=a(x-y)$ for all $x,y \in \Omega$.
In addition, the condition (\ref{contkernel}) is also satisfied by some more general choices of kernel coefficients, for example if
$$A(x,y)=A^\prime(x,y)A_0(x,y),$$
where $A^\prime \in \mathcal{L}_0(\lambda^\frac{1}{2})$ is continuous in $\Omega \times \Omega$ and $A_0 \in \mathcal{L}_0(\lambda^\frac{1}{2})$ is translation invariant inside of $\Omega$, but is not required to satisfy any continuity or smoothness assumption. Furthermore, we stress that the condition (\ref{contkernel}) only restricts the behaviour of $A$ inside of $\Omega \times \Omega$, while outside of $\Omega \times \Omega$ a more general behaviour is possible. \newline
We are now in the position to state our main result.
\begin{theorem} \label{mainint5}
	Let $\Omega \subset \mathbb{R}^n$ be a domain, $s \in (0,1)$, $\lambda \geq 1$ and $p \in (2,\infty)$. Moreover, let $g:\mathbb{R}^n \times \mathbb{R}^n \to \mathbb{R}$ be a measurable symmetric function with $\nabla^s g \in L^p_{loc}(\Omega)$ and assume that $f \in L^{p_\star}_{loc}(\Omega)$, where $p_\star=\max \left \{\frac{np}{n+sp},2 \right \}$. 
	If $A \in \mathcal{L}_0(\lambda)$ satisfies the condition (\ref{contkernel}) and if $\Phi$ satisfies the conditions (\ref{PhiLipschitz}) and (\ref{PhiMonotone}) with respect to $\lambda$, then for $F$ given as in $(\ref{F})$ and any weak solution $u \in H^{s,2}_{loc}(\Omega | \mathbb{R}^n)$ 
	of the equation
	$$
	L_A^\Phi u =  F \text{ in } \Omega,
	$$
	we have $u \in H^{s,p}_{loc}(\Omega|\mathbb{R}^n)$. \newline
	Moreover, for all open sets $U \Subset V \Subset \Omega$, we have
	\begin{equation} \label{Sobest}
	||\nabla^s u||_{L^p(U)} \leq C \left (||f||_{L^{p_\star}(V)}  + ||\nabla^s g||_{L^p(V)} + ||\nabla^s u||_{L^2(V)} \right ),
	\end{equation}
	where $C=C(p,n,s,\lambda,U,V)>0$.
\end{theorem}

\begin{remark} \normalfont
	In view of (\ref{relationsSob}), under the assumptions of Theorem \ref{mainint5}, weak solutions of (\ref{nonlocaleq}) in particular belong to the Bessel potential space $H^{s,p}_{loc}(\Omega)$ and also to the Sobolev-Slobodeckij space $W^{s,p}_{loc}(\Omega)$.
	Moreover, the condition $\nabla^s g \in L^p_{loc}(\Omega)$ is for example satisfied if $g$ has the form 
	\begin{equation} \label{or}
	g(x,y) = \sum_{i=1}^m D_i(x,y)(g_i(x)-g_i(y)),
	\end{equation}
	where $m \in \mathbb{N}$, $D_i \in L^\infty(\mathbb{R}^n \times \mathbb{R}^n)$ and $g_i \in H^{s,p}_{loc}(\Omega|\mathbb{R}^n)$ for all $i=1,...,m$. By (\ref{relationsSob}), the latter condition is in particular satisfied if all $g_i$ belong to the Bessel potential space $H^{s,p}(\mathbb{R}^n)$.
\end{remark}

\begin{remark} \normalfont
	An interesting feature of the estimate (\ref{Sobest}) is that it is not a purely local estimate, in the sense that due to the nonlocal nature of the s-gradient $\nabla^s$, the left-hand side also depends on the values of $u$ outside the domain $\Omega$. In other words, we also gain some control on $u$ outside the domain where the equation holds.
\end{remark}

For the sake of providing some context, let us briefly consider local elliptic equations in divergence form of the type
\begin{equation} \label{localeq}
\textnormal{div}(B \nabla u)=\textnormal{div} h + f \quad \text{in } \Omega,
\end{equation}
where the matrix of coefficients $B=\{b_{ij}\}_{i,j=1}^n$ is assumed to be uniformly elliptic and bounded, while $h:\mathbb{R}^n \to \mathbb{R}^n$ and $f:\mathbb{R}^n \to \mathbb{R}$ are given functions. The equation (\ref{localeq}) can in some sense be thought of as a local analogue of the nonlocal equation (\ref{nonlocaleq}) corresponding to the limit case $s=1$. It is known that if the coefficients $b_{ij}$ are continuous and $h \in L^p_{loc}(\Omega,\mathbb{R}^n)$, $f \in L^\frac{np}{n+p}_{loc}(\Omega)$ for some $p>2$, then weak solutions $u \in W^{1,2}_{loc}(\Omega)$ of the equation (\ref{localeq}) belong to $W^{1,p}_{loc}(\Omega)$. This corresponds to our main result in the sense that we obtain local $W^{s,p}$ regularity for nonlocal equations of the type (\ref{nonlocaleq}) in the case when $A$ satisfies the continuity assumption (\ref{contkernel}). We note that this $W^{1,p}_{loc}(\Omega)$ regularity for solutions of the equation (\ref{localeq}) also holds if more generally the coefficients $b_{ij}$ belong to the space VMO of functions with vanishing mean oscillation, cf. \cite{ByunLp} or \cite{AM}. Therefore, an interesting question is if the conclusion of Theorem \ref{mainint5} remains true for kernel coefficients $A$ that belong to VMO in a suitable sense. \newline
Regarding related previous results, in \cite{Me} Theorem \ref{mainint5} was proved in the linear case when $\Phi(t)=t$ and under the stronger assumption that $A$ is translation invariant in the whole space $\mathbb{R}^n$ and in the special case when $g$ is of the form (\ref{or}). Another very interesting result in this direction was recently proved in \cite{MSY}, where again in the linear case when $\Phi(t)=t$ it was in particular shown that if $A \in \mathcal{L}_0(\lambda)$ is H\"older continuous with some arbitrary H\"older exponent and for some $2 \leq p<\infty$ we have $f \in L^p(\mathbb{R}^n)$, then weak solutions $u \in H^{s,2}(\mathbb{R}^n)$ of the equation $L_A^\Phi u= f$ in $\mathbb{R}^n$ belong to $H^{\alpha,p}_{loc}(\mathbb{R}^n)$ for any $\alpha<\min \big \{2s,1 \big\}$, gaining not only integrability, but also differentiability, while for local equations of the type (\ref{localeq}) no comparable gain of differentiability is attainable. Another interesting question is therefore if such a gain of differentiability is also achievable for possibly nonlinear equations of the type (\ref{nonlocaleq}) that might only hold in some domain $\Omega$ with kernel coefficients that satisfy the condition (\ref{contkernel}) or even for kernels of VMO-type. We plan to investigate this direction in the future. \newline
More results concerning Sobolev regularity for nonlocal equations are for example proved in \cite{Warma}, \cite{KassMengScott}, \cite{selfimpro}, \cite{Schikorra}, \cite{MP}, \cite{BL}, \cite{Cozzi}, while various results on H\"older regularity are proved in \cite{Fall}, \cite{Fall1}, \cite{MeH}, \cite{NonlocalGeneral}, \cite{CSa}, \cite{Grubb}, \cite{finnish}, \cite{Kassmann}, \cite{Silvestre} and \cite{Peral}. Furthermore, for some regularity results concerning nonlocal equations similar to (\ref{nonlocaleq}) in the more general setting of measure data, we refer to \cite{mdata}.
\subsection{Approach}
Our approach is inspired by an approach introduced by Caffarelli and Peral in \cite{CaffarelliPeral} in the context of obtaining $W^{1,p}$ estimates for local elliptic equations of the type (\ref{localeq}). The philosophy of the approach is as follows. The first step is to locally approximate the gradient a weak solutions $u$ of (\ref{localeq}) by the gradient of a weak solution $v$ to a suitable homogeneous equation for which an in some sense good enough estimate is already known. More presicely, in the context of local equations, one exploits the fact that the approximate solution $v$ is already known to satisfy a local $C^{0,1}$ estimate in order to transfer some regularity to $u$. In fact, a real-variable argument based on the Vitali covering lemma, the Hardy-Littlewood maximal function and an alternative characterization of $L^p$ spaces then allows to prove an $L^p$ estimate for the gradient $\nabla u$ corresponding to our estimate (\ref{Sobest}), which then implies the desired local $W^{1,p}$ estimate. \newline
The main idea in order to prove Theorem \ref{mainint5} is to apply a similar strategy with the gradient $\nabla u$ replaced by the nonlocal s-gradient $\nabla^s u$. In particular, in our nonlocal setting the local $C^{0,1}$ estimate for the approximate solution has to be replaced by a local $C^{s+\gamma}$ estimate for some $\gamma>0$. Such an estimate was recently proved in \cite{MeH} for equations of the type (\ref{nonlocaleq}) with kernel coefficients that satisfy the condition (\ref{contkernel}), opening the way towards obtaining our Theorem \ref{mainint5}. This estimate is used in an adaptation of the real-variable argument described above in order to obtain the desired estimate (\ref{Sobest}) from Theorem \ref{mainint5}. In contrast to \cite{Me}, additional difficulties also arise due to the presense of the nonlinearity $\Phi$, which are dealt with by careful applications of the conditions (\ref{PhiLipschitz}) and (\ref{PhiMonotone}) throughout the paper and using the theory of monotone operators in order to prove existence and uniqueness for the corresponding Dirichlet problem.

\section{Preliminaries}
\subsection{Some notation}
For convenience, let us fix some notation which we use throughout the paper. By $C$ and $C_i$, $i \in \mathbb{N}$, we always denote positive constants, while dependences on parameters of the constants will be shown in parentheses. As usual, by
$$ B_r(x_0):= \{x \in \mathbb{R}^n \mid |x-x_0|<r \}$$
we denote the open ball with center $x_0 \in \mathbb{R}^n$ and radius $r>0$. Moreover, if $E \subset \mathbb{R}^n$ is measurable, then by $|E|$ we denote the $n$-dimensional Lebesgue-measure of $E$. If $0<|E|<\infty$, then for any $u \in L^1(E)$ we define
$$ \overline u_{E}:= \dashint_{E} u(x)dx := \frac{1}{|E|} \int_{E} u(x)dx.$$

\subsection{Some tools from real analysis}

In this section, we discuss some results from real analysis that are at the core of the real-variable argument mentioned in section 1.2. \newline
The following result is an application of the well-known Vitali covering lemma, cf. \cite[Theorem 2.7]{ByunLp}.

\begin{lemma} \label{modVitali}
	Assume that $E$ and $F$ are measurable sets in $\mathbb{R}^n$ that satisfy $E \subset F \subset B_1$. Assume further that 
	there exists some $\varepsilon \in (0,1)$ such that 
	$$
	|E| < \varepsilon |B_1|,
	$$
	and that for all $x \in B_1$ and any $r \in (0,1)$ with $|E \cap B_r(x)| \geq \varepsilon |B_r(x)|$ we have 
	$$B_r(x) \cap B_1 \subset F. $$
	Then we have 
	$$
	|E| \leq 10^n \varepsilon |F|.
	$$
\end{lemma}

Another tool we use is the Hardy-Littlewood maximal function.

\begin{defin}
	Let $f \in L^1_{loc}(\mathbb{R}^n)$. Then the Hardy-Littlewood maximal function \newline $\mathcal{M} f: \mathbb{R}^n \to [0,\infty]$ of $f$ is defined by 
	$$ \mathcal{M} f(x):=\mathcal{M} (f)(x) := \sup_{\rho>0} \dashint_{B_\rho(x)} |f(y)|dy .$$
	Moreover, for any domain $\Omega \subset \mathbb{R}^n$ and any function $f \in L^1(\Omega)$, consider the zero extension of $f$ to $\mathbb{R}^n$  
	$$ f_{\Omega} (x) := \begin{cases} 
	f(x) \text{, if } x \in \Omega \\
	0 \quad  \text{ , if } x \notin \Omega.
	\end{cases} $$
	We then define 
	$$ \mathcal{M}_{\Omega} f := \mathcal{M} f_\Omega. $$
\end{defin}

The following Lemma contains the scaling and translation invariance of the Hardy-Littlewood maximal function and can be proved by using a change of variables.
\begin{lemma} \label{scale}
	Let $f \in L^1_{loc}(\mathbb{R}^n)$, $r>0$ and $y \in \mathbb{R}^n$. Then for the function $f_{r,y}(x):=f(rx+y)$ and any $x \in \mathbb{R}^n$ we have 
	$$\mathcal{M} f_{r,y}(x) = \mathcal{M} f(rx+y) .$$
	Similarly, for any domain $\Omega \subset \mathbb{R}^n$, any function $f \in L^1(\Omega)$ and any $x \in \Omega$ we have 
	$$\mathcal{M}_{ \Omega^\prime} f_{r,y}(x) = \mathcal{M}_\Omega f(rx+y), $$
	where $\Omega^\prime:= \{ \frac{x-y}{r} \mid x \in \Omega \}$. 
\end{lemma}
We remark that for any $f \in L^1_{loc}(\mathbb{R}^n)$, $\mathcal{M} f$ is Lebesgue-measurable. \newline
The probably most important properties of the Hardy-Littlewood maximal function are contained in the following result, see \cite{SteinSingular}.
\begin{proposition} \label{Maxfun} 
	Let $\Omega \subset \mathbb{R}^n$ be a domain.
	\begin{enumerate}
		\item (weak 1-1 estimate) If $f \in L^1(\Omega)$ and $t>0$, then 
		$$
		| \{x \in \Omega \mid \mathcal{M}_\Omega(f)(x) > t \}| \leq \frac{C}{t} \int_{\Omega} |f| dx, 
		$$
		where $C=C(n)>0$.
		\item (strong p-p estimates) If $f \in L^p(\Omega)$ for some $p \in (1,\infty]$, then 
		$$
		||f||_{L^p(\Omega)} \leq ||\mathcal{M}_\Omega f||_{L^p(\Omega)} \leq C ||f||_{L^p(\Omega)}, 
		$$
		where $C=C(n,p)>0$.
		\item If $f \in L^p(\Omega)$ for some $p \in [1,\infty]$, then the function $\mathcal{M}_\Omega f$ is finite almost everywhere.
	\end{enumerate}
\end{proposition}
We conclude this section by giving an alternative characterization of $L^p$ spaces, see \cite[Lemma 7.3]{CaffFully}. It can be proved by using the well-known formula
$$ ||f||^p_{L^p(\Omega)} = p \int_0^{\infty} t^{p-1} \left | \left \{ x \in \Omega \mid f(x) > t \right \} \right | dt. $$
\begin{lemma} \label{Caff}
	Let $0< p<\infty$. Furthermore, suppose that $f$ is a nonnegative and measurable function in a bounded domain $\Omega \subset \mathbb{R}^n$ and let $\tau > 0$, $\beta>1$. Then for
	$$ S:= \sum_{k = 1}^{\infty}{\beta^{kp} | \{ x \in \Omega \mid f(x) > \tau \beta^k \} | } , $$
	we have
	$$
	C^{-1} S \leq ||f||^p_{L^p(\Omega)} \leq C (|\Omega|+S)
	$$
	for some constant $C=C(\tau,\beta,p) > 0$.
	In particular, we have $f \in L^p(\Omega)$ if and only if $S<\infty$.
\end{lemma}

\subsection{Fractional Sobolev spaces}
The following type of fractional Sobolev spaces is probably the most common type of such spaces in the literature concerned with nonlocal equations similar to (\ref{nonlocaleq}).
\begin{defin}
	Let $\Omega \subset \mathbb{R}^n$ be a domain. For $p \in [1,\infty)$ and $s \in (0,1)$, we define the Sobolev-Slobodeckij space
	$$W^{s,p}(\Omega):=\left \{u \in L^p(\Omega) \mathrel{\Big|} \int_{\Omega} \int_{\Omega} \frac{|u(x)-u(y)|^p}{|x-y|^{n+sp}}dydx<\infty \right \}$$
	with norm
	$$ ||u||_{W^{s,p}(\Omega)} := \left (\int_{\Omega}|u(x)|^p dx + \int_{\Omega} \int_{\Omega} \frac{|u(x)-u(y)|^p}{|x-y|^{n+sp}}dydx \right )^{1/p} .$$
	Moreover, we also define the corresponding local versions of these spaces by 
	$$ W^{s,p}_{loc}(\Omega) := \left \{ u \in L^p_{loc}(\Omega) \mid u \in W^{s,p}(\Omega^\prime) \text{ for any domain } \Omega^\prime \Subset \Omega \right \}.$$   
	In addition, we also use the space
	\begin{align*}
	W_0^{s,2}(\Omega):= & \left \{u \in W^{s,2}(\mathbb{R}^n) \mid u=0 \text{ in } \mathbb{R}^n \setminus \Omega \right \} \\
	= & \left \{u \in H^{s,2}(\Omega | \mathbb{R}^n) \mid u=0 \text{ in } \mathbb{R}^n \setminus \Omega \right \} .
	\end{align*}
\end{defin}
\begin{remark} \label{Hilbert} \normalfont
	The space $W^{s,2}(\Omega)$ is a separable Hilbert space with respect to the inner product
	$$
	(u,v)_{W^{s,2}(\Omega)}:= (u,v)_{L^2(\Omega)} + \int_{\Omega} \int_{\Omega} \frac{(u(x)-u(y))(v(x)-v(y))}{|x-y|^{n+2s}}dydx. 
	$$
	Furthermore, the space $W_0^{s,2}(\Omega)$ clearly is a closed subspace of $W^{s,2}(\mathbb{R}^n)$ and is therefore also a separable Hilbert space with respect to the inner product $(\cdot,\cdot)_{W^{s,2}(\mathbb{R}^n)}$.
\end{remark}
We often use the following fractional Poincar\'e-type inequalities.
\begin{lemma} \label{Poincare} (fractional Poincar\'e inequality)
	Let $s \in (0,1)$ and $R>0$. For any $u \in W^{s,2}(B_R)$, we have
	$$ \int_{B_R} \left | u(x)- \overline u_{B_R} \right |^2 dx \leq C R^{2s} \int_{B_R} \int_{B_R} \frac{|u(x)-u(y)|^2}{|x-y|^{n+2s}}dydx,$$
	where $C=C(n,s)>0$.
\end{lemma}
\begin{proof}
	Using Jensen's inequality, for any $x \in B_R$ we obtain
	\begin{align*}
	\left | u(x)- \overline u_{B_R} \right |^2 \leq \left (\dashint_{B_R} |u(x)-u(y)|dy \right)^2 \leq & \dashint_{B_R} |u(x)-u(y)|^2dy \\
	\leq & C R^{2s} \int_{B_R} \frac{|u(x)-u(y)|^2}{|x-y|^{n+2s}}dy,
	\end{align*}
	where $C=C(n,s)>0$. The claim now follows by integrating both sides over $B_R$.
\end{proof}

For a proof of the following inequality we refer to \cite[Lemma 2.3]{MeH}.
\begin{lemma} \label{Friedrichsx} (fractional Friedrichs-Poincar\'e inequality)
	Let $s \in(0,1)$ and consider a bounded domain $\Omega \subset \mathbb{R}^n$. For any $u \in W^{s,2}_0(\Omega)$, we have
	\begin{equation} \label{FPI4}
	\int_{\mathbb{R}^n} |u(x)|^2 dx \leq C |\Omega|^{\frac{2s}{n}} \int_{\mathbb{R}^n} \int_{\mathbb{R}^n} \frac{|u(x)-u(y)|^2}{|x-y|^{n+2s}}dydx,
	\end{equation}
	where $C=C(n,s)>0$.
\end{lemma}
We also use the following type of fractional Sobolev spaces.
\begin{defin}
	For $p \in [1,\infty)$ and $s \in \mathbb{R}$, consider the Bessel potential space
	$$ H^{s,p}(\mathbb{R}^n):=\left \{u \in L^p(\mathbb{R}^n) \mid \mathcal{F}^{-1} \left [ \left ( 1+|\xi|^2 \right )^{\frac{s}{2}} \mathcal{F} f \right ] \in L^p(\mathbb{R}^n) \right \} ,$$
	where $\mathcal{F}$ denotes the Fourier transform and $\mathcal{F}^{-1}$ denotes the inverse Fourier transform. We equip $H^{s,p}(\mathbb{R}^n)$ with the norm
	$$ ||u||_{H^{s,p}(\mathbb{R}^n)} := \left | \left |\mathcal{F}^{-1} \left [ \left ( 1+|\xi|^2 \right )^{\frac{s}{2}} \mathcal{F} f \right ] \right | \right |_{L^p(\mathbb{R}^n)}.$$
	Moreover, for any domain $\Omega \subset \mathbb{R}^n$ we define 
	$$H^{s,p}(\Omega):= \left \{ v \big|_\Omega \mid v \in H^{s,p}(\mathbb{R}^n) \right \} $$
	with norm 
	$$ ||u||_{H^{s,p}(\Omega)} := \inf \left \{ ||v||_{H^{s,p}(\mathbb{R}^n)} \mid v \big |_\Omega = u \right \}$$
	and also the corresponding local Bessel potential spaces by 
	$$ H^{s,p}_{loc}(\Omega) := \left \{ u \in L^p_{loc}(\Omega) \mid u \in H^{s,p}(\Omega^\prime) \text{ for any domain } \Omega^\prime \Subset \Omega \right \}.$$   
\end{defin}
The following result gives some relations between Bessel potential spaces and Sobolev-Slobodeckij spaces.

\begin{proposition} \label{Sobolevrelations}
	Let $\Omega \subset \mathbb{R}^n$ be a domain.
	\begin{enumerate}
		\item If $\Omega$ is a bounded Lipschitz domain or $\Omega=\mathbb{R}^n$, then for all $s \in (0,1)$, $p \in (1,2]$ we have $W^{s,p}(\Omega) \hookrightarrow H^{s,p}(\Omega)$.
		\item For any $s \in (0,1)$ and any $p \in [2,\infty)$ we have $ H^{s,p}(\Omega) \hookrightarrow W^{s,p}(\Omega)$.
	\end{enumerate}
\end{proposition}
For a proof of Proposition \ref{Sobolevrelations} in the case when $\Omega = \mathbb{R}^n$, we refer to Theorem 5 in chapter \RNum{5} of \cite{SteinSingular}. For a brief explanation on how to obtain the result for general domains, we refer to \cite[section 3]{Me}.

We now generalize the notion of the $s$-gradient which was introduced in the introduction.
\begin{defin}
	Let $s \in (0,1)$. For any domain $\Omega \subset \mathbb{R}^n$ and any measurable function $u:\Omega \to \mathbb{R}$, we define the s-gradient $\nabla^s_\Omega u:\Omega \to [0,\infty]$ by 
	$$ \nabla^s_\Omega u(x):= \left ( \int_{\Omega} \frac{(u(x)-u(y))^2}{|x-y|^{n+2s}}dy \right )^{\frac{1}{2}}.$$
\end{defin}
In particular, note that we have $\nabla^s u=\nabla^s_{\mathbb{R}^n} u$.
As mentioned in the introduction, the notion of the $s$-gradient is closely related with the Bessel potential spaces $H^{s,p}$. The precise relation is given by the following result. 

\begin{proposition} \label{altcharBessel}
	Let $s \in (0,1)$, $p \in \left (\frac{2n}{n+2s},\infty \right )$ and assume that $\Omega \subset \mathbb{R}^n$ is a bounded Lipschitz domain or that $\Omega=\mathbb{R}^n$. Then we have $u \in H^{s,p}(\Omega)$ if and only if $u \in L^p(\Omega)$ and $\nabla^s_\Omega u \in L^p(\Omega)$. Moreover, we have
	$$ ||u||_{H^{s,p}(\Omega)} \simeq ||u||_{L^p(\Omega)} + ||\nabla^s_\Omega u||_{L^p(\Omega)} $$
	in the sense of equivalent norms.
\end{proposition}
This characterization was first given by Stein in \cite{SteinBessel} in the case when $\Omega=\mathbb{R}^n$. For the case when $\Omega$ is a bounded Lipschitz domain we refer to \cite[Theorem 1.3]{Soler}, where this characterization is proved in the more general context of Triebel-Lizorkin spaces and so-called uniform domains.
\begin{remark} \normalfont
	In view of Proposition \ref{altcharBessel} and Proposition \ref{Sobolevrelations}, for any bounded Lipschitz domain $\Omega \subset \mathbb{R}^n$ and all $s \in (0,1)$, $p \in \left [2,\infty \right )$ we have the inclusions
	$$ H^{s,p}(\mathbb{R}^n) \subset H^{s,p}(\Omega | \mathbb{R}^n) \subset H^{s,p}(\Omega) \subset W^{s,p}(\Omega).$$
	In the case when $\Omega \subset \mathbb{R}^n$ is an arbitrary domain this implies the inclusions (\ref{relationsSob}) from the introduction.
\end{remark}
We also use the following standard embedding theorems of Bessel potential spaces. For precise references see \cite[section 3]{Me}.
\begin{theorem} \label{BesselEmbedding}
	Let $1<p \leq p_1 < \infty$, $s,s_1 \geq 0$ and assume that $\Omega \subset \mathbb{R}^n$ is a domain.
	\begin{enumerate}
		\item If $sp<n$, then for any $q \in [p,\frac{np}{n-sp}]$ we have
		$$ H^{s,p}(\Omega) \hookrightarrow L^{q}(\Omega).$$
		\item More generally, if $ s - \frac{n}{p} = s_1 - \frac{n}{p_1}, $ then
		$$ H^{s,p}(\Omega) \hookrightarrow H^{s_1,p_1}(\Omega).$$
		\item If $sp = n$, then for any $q \in [p,\infty)$ we have
		$$ H^{s,p}(\Omega) \hookrightarrow L^{q}(\Omega).$$
		\item If $sp > n$, then we have $$H^{s,p}(\Omega) \hookrightarrow C^\alpha(\Omega),$$
		where $\alpha = s-\frac{n}{p}$.
	\end{enumerate}
\end{theorem}
\section{Some preliminary estimates}
For the rest of this paper, we fix real numbers $s \in (0,1)$ and $\lambda \geq 1$.

\subsection{Tail estimates}
The following Lemma relates the tails of a function to the $L^2$ norm of its $s$-gradient. For a proof we refer to \cite[Lemma 4.1]{Me}.
\begin{lemma} \label{tailestimate}
	For all $r,R>0$ and any $u \in H^{s,2}(B_R | \mathbb{R}^n)$ we have
	\begin{equation} \label{cool}
	\int_{\mathbb{R}^n \setminus B_r} \frac{u(y)^2}{|y|^{n+2s}}dy \leq C(||\nabla^s u||_{L^2(B_R)}^2+||u||_{L^2(B_R)}^2) ,
	\end{equation}
	where $C=C(n,s,r,R)>0$.
\end{lemma}
Finally, the following result can be proved in the same way as \cite[Corollary 4.4]{Me}, by using the $L^\infty$ estimate from \cite[Theorem 2.11]{MeH} instead of the one from \cite[Theorem 4.2]{Me}. It shows that that if a function satisfies a homogeneous nonlocal equation, then the tails of its $s$-gradient can be controlled nicely, so that we can focus on estimating the local part of the $s$-gradient.
\begin{proposition} \label{taillinf}
	Consider a kernel coefficient $A \in \mathcal{L}_0(\lambda)$ and assume that $\Phi$ satisisfies (\ref{PhiLipschitz}) and (\ref{PhiMonotone}). Then for all $0<r<R$ and any weak solution $u \in H^{s,2}(B_R | \mathbb{R}^n)$ of the equation 
	$$ L_A^\Phi u = 0 \text{ in } B_R,$$
	we have the estimate
	\begin{equation} \label{gradest45}
	||\nabla^s_{\mathbb{R}^n \setminus B_R} u||_{L^\infty(B_r)} \leq C||\nabla^s u||_{L^2(B_R)},
	\end{equation}
	where $C=C(n,s,r,R,\lambda)>0$.
\end{proposition}

\subsection{Higher H\"older regularity}
In the basic case when $A \in \mathcal{L}_0(\lambda)$, it is known that any weak solution to a corresponding homogeneous nonlocal equation is locally $C^\alpha$ for some $\alpha>0$, cf. \cite[Theorem 1.2]{finnish}. 
The following result shows that if $A \in \mathcal{L}_0(\lambda)$ additionally satisfies the condition (\ref{contkernel}), then such weak solutions enjoy better H\"older regularity than in general.
\begin{proposition} \label{modC2sreg}
	Consider a kernel coefficient $A \in \mathcal{L}_0(\lambda)$ that satisfies the condition (\ref{contkernel}) in $B_5$ and suppose that $\Phi$ satisfies (\ref{PhiLipschitz}) and (\ref{PhiMonotone}) with respect to $\lambda$. Moreover, assume that $u \in H^{s,2}(B_5 | \mathbb{R}^n)$ is a weak solution of the equation $L_{A}^\Phi u = 0$ in $B_5$. Then for any $0<\alpha<\min\{2s,1\}$, we have
	$$ [u]_{C^{\alpha}(B_3)} \leq C ||\nabla^s u||_{L^2(B_5)} ,$$
	where $C=C(n,s,\lambda,\alpha)>0$ and 
	$$[u]_{C^{\alpha}(B_3)}:=\sup_{\substack{_{x,y \in B_3}\\{x \neq y}}} \frac{|u(x)-u(y)|}{|x-y|^{\alpha}}.$$
\end{proposition}
We will derive Proposition \ref{modC2sreg} from Theorem \ref{HiHol} below, which is proved in \cite[Theorem 1.1]{MeH}. In order to state the result, we need the following definitions. First, we define the tail space
$$L^1_{2s}(\mathbb{R}^n):= \left \{u \in L^1_{loc}(\mathbb{R}^n) \mathrel{\Big|} \int_{\mathbb{R}^n} \frac{|u(y)|}{1+|y|^{n+2s}}dy < \infty \right \}.$$
The most important property of this space is that for any function $u \in L^1_{2s}(\mathbb{R}^n)$, the quantity 
$$ \int_{\mathbb{R}^n \setminus B_{R}(x_0)} \frac{|u(y)|}{|x_0-y|^{n+2s}}dy$$
is finite for all $R>0$, $x_0 \in \mathbb{R}^n$. 
\begin{defin}
	We say that $u \in W^{s,2}_{loc}(\Omega) \cap L^1_{2s}(\mathbb{R}^n)$ is a local weak solution of the equation $L_A^\Phi u =0$ in $\Omega$, if 
	\begin{equation} \label{weaksolx}
	\mathcal{E}_A^\Phi(u,\varphi) =  0 \quad \forall \varphi \in H_c^{s,2}(\Omega).
	\end{equation}
\end{defin}

\begin{theorem} \label{HiHol}
	Let $\Omega \subset \mathbb{R}^n$ be a domain. Consider a kernel coefficient $A \in \mathcal{L}_0(\lambda)$ that satisfies the condition (\ref{contkernel}) in $\Omega$ and suppose that $\Phi$ satisfies (\ref{PhiLipschitz}) and (\ref{PhiMonotone}) with respect to $\lambda$. Moreover, assume that $u \in W^{s,2}_{loc}(\Omega) \cap L^1_{2s}(\mathbb{R}^n)$ is a local weak solution of the equation $L_{A}^\Phi u = 0$ in $\Omega$. Then for any $0<\alpha<\min \big \{2s,1 \big\}$, we have $u \in C^\alpha_{loc}(\Omega)$. \newline Furthermore, for all $R>0$, $x_0 \in \mathbb{R}^n$ such that $B_R(x_0) \Subset \Omega$ and any $\sigma \in (0,1)$, we have
	\begin{equation} \label{Hoeldest}
	[u]_{C^\alpha(B_{\sigma R}(x_0))} \leq \frac{C}{R^\alpha} \bigg ( R^{-\frac{n}{2}} ||u||_{L^2(B_R(x_0))} + R^{2s} \int_{\mathbb{R}^n \setminus B_{R}(x_0)} \frac{|u(y)|}{|x_0-y|^{n+2s}}dy \bigg ),
	\end{equation}
	where $C=C(n,s,\lambda,\alpha,\sigma)>0$.
\end{theorem}

In order to derive Proposition \ref{modC2sreg} from Theorem \ref{HiHol}, we need to ensure that as the terminology suggests, any weak solution as defined in the introduction is also a local weak solution. This is essentially a consequence of the following Lemma.
\begin{lemma} \label{tailest2}
	Let $R>0$ and $s \in (0,1)$. Then for any function $u \in H^{s,2}(B_R | \mathbb{R}^n)$ and any $R>0$, we have $u \in L^1_{2s}(\mathbb{R}^n)$ and
	$$ \int_{\mathbb{R}^n} \frac{|u(y)|}{1+|y|^{n+2s}} dy \leq C \left ( ||u||_{L^2(B_R)} +||\nabla^s u||_{L^2(B_R)} \right ) ,$$
	where $C=C(n,s,R)>0$. In particular, we have $H^{s,2}(B_R | \mathbb{R}^n) \subset L^1_{2s}(\mathbb{R}^n)$.
\end{lemma}
\begin{proof}
	First of all, integration in polar coordinates yields
	\begin{equation} \label{intpolar1}
	\int_{\mathbb{R}^n \setminus B_R} \frac{dz}{|z|^{n+2s}}= C_1 R^{-2s},
	\end{equation}
	where $C_1=C_1(n,s)>0$.
	We split the integral in question as follows
	\begin{align*}
	\int_{\mathbb{R}^n} \frac{|u(y)|}{1+|y|^{n+2s}} dy \leq & \int_{B_R} |u(y)|dy + \int_{\mathbb{R}^n \setminus B_R} \frac{|u(y)|}{|y|^{n+2s}} dy \\
	\leq & C_2 \left ( \int_{B_R} |u(y)|^2dy \right)^\frac{1}{2} + C_3 \left ( \int_{\mathbb{R}^n \setminus B_R} \frac{|u(y)|^2}{|y|^{n+2s}} dy \right)^\frac{1}{2},
	\end{align*}
	where $C_2=C_2(n,R)>0$ and $C_3=C_3(n,s,R)>0$. Here we used the Cauchy-Schwarz inequality and (\ref{intpolar1}) in order to obtain the last inequality. In view of Lemma \ref{tailestimate}, we also have
	$$ \left ( \int_{\mathbb{R}^n \setminus B_R} \frac{|u(y)|^2}{|y|^{n+2s}} dy \right)^\frac{1}{2} \leq C_4 \left ( ||u||_{L^2(B_R)} +||\nabla^s u||_{L^2(B_R)} \right ),$$
	where $C_4=C_4(n,s,R)>0$. The claim now follows by combining the above two estimates.
\end{proof}

\begin{proof}[Proof of Proposition \ref{modC2sreg}]
	Since in view of Lemma \ref{tailest2} the function $u_0:=u- \overline u_{B_5} \in H^{s,2}(B_5 | \mathbb{R}^n) \subset W^{s,2}_{loc}(B_5) \cap L^1_{2s}(\mathbb{R}^n)$ is a local weak solution of 
	$$ L_A^\Phi u_0=0 \text{ in } B_5,$$ by Theorem \ref{HiHol}, (\ref{intpolar1}), Lemma \ref{tailestimate} and the fractional Poincar\'e inequality (Lemma \ref{Poincare}), we have
	\begin{align*}
	[u]_{C^{\alpha}(B_3)} & = [u_0]_{C^{\alpha}(B_3)} \\
	& \leq C_1 \left (||u_0||_{L^2(B_4)} + \int_{\mathbb{R}^n \setminus B_4} \frac{|u_0(y)|}{|y|^{n+2s}}dy \right )\\
	& \leq C_2 \left (||u_0||_{L^2(B_5)} + \left (\int_{\mathbb{R}^n \setminus B_4} \frac{|u_0(y)|^2}{|y|^{n+2s}}dy \right )^\frac{1}{2} \right )\\
	& \leq C_3 \left (||u_0||_{L^2(B_5)} + ||\nabla^s u||_{L^2(B_5)} \right ) \leq C_4||\nabla^s u||_{L^2(B_5)},
	\end{align*}
	where all constants depend only on $n,s,\lambda$ and $\alpha$. This finishes the proof.
\end{proof}

\section{The Dirichlet problem}
In this section, we are mainly concerned with the existence and uniqueness of weak solutions to nonlocal Dirichlet problems. Although in this paper we only use the existence of weak solutions in the space $H^{s,2}(\Omega |\mathbb{R}^n)$, for future reference and for the sake of generality we also include some other solution spaces. \newline
Throughout this section, we fix a bounded domain $\Omega \subset \mathbb{R}^n$ and let $X$ be a vector space that satisfies
\begin{equation} \label{spaces}
W^{s,2}(\mathbb{R}^n) \subset X \subset H^{s,2}(\Omega | \mathbb{R}^n).
\end{equation}
In particular, possible choices for $X$ are $X=H^{s,2}(\Omega | \mathbb{R}^n)$ and $X=W^{s,2}(\mathbb{R}^n)$.
\begin{defin}
	Suppose that $X$ satisfies (\ref{spaces}). Moreover, let $h \in X$ and $f \in L^\frac{2n}{n+2s}(\Omega)$. We say that $u \in X$ is a weak solution of the problem \begin{equation} \label{constcof31}
	\begin{cases} \normalfont
	L_A^\Phi u = f & \text{ in } \Omega \\
	u = h & \text{ a.e. in } \mathbb{R}^n \setminus \Omega,
	\end{cases}
	\end{equation}
	if we have
	$\mathcal{E}_A^\Phi(u,\varphi) = (f,\varphi)_{L^2(\Omega)}$ for all $\varphi \in W^{s,2}_0(\Omega)$
	and $u = h \text{ a.e. in } \mathbb{R}^n \setminus \Omega$.
\end{defin}

\begin{proposition} \label{Dirichlet}
	Let $\Omega \subset \mathbb{R}^n$ be a bounded domain and suppose that $X$ is a vector space that satisfies (\ref{spaces}). 
	Consider a kernel coefficient $A \in \mathcal{L}_0(\lambda)$ and suppose that $\Phi$ satisfies (\ref{PhiLipschitz}) and (\ref{PhiMonotone}). Moreover, let $h \in X$ and $f \in L^\frac{2n}{n+2s}(\Omega)$. Then there exists a unique weak solution $u \in X$ of the Dirichlet problem (\ref{constcof31}).
\end{proposition}

\begin{proof}
	We use an argument inspired by \cite{existence} based on the theory of monotone operators.
	Fix $h \in X$ and consider the operator $\mathcal{A}:W_0^{s,2}(\Omega) \to (W_0^{s,2}(\Omega))^\star$ defined by 
	\begin{align*}
	\langle \mathcal{A}(v),\varphi \rangle := \langle \mathcal{A}_1(v),\varphi \rangle + \langle \mathcal{A}_2(v),\varphi \rangle,
	\end{align*}
	where 
	$$ \langle \mathcal{A}_1(v),\varphi \rangle := \int_{\Omega} \int_{\mathbb{R}^n} \frac{A(x,y)}{|x-y|^{n+2s}} \Phi(v(x)+h(x)-v(y)-h(y))(\varphi(x)-\varphi(y))dydx $$
	and
	$$ \langle \mathcal{A}_2(v),\varphi \rangle := \int_{\mathbb{R}^n \setminus \Omega} \int_{\Omega} \frac{A(x,y)}{|x-y|^{n+2s}} \Phi(v(x)+h(x)-v(y)-h(y))(\varphi(x)-\varphi(y))dydx. $$
	Here by $(W_0^{s,2}(\Omega))^\star$ we denote the dual space of $W_0^{s,2}(\Omega)$ consisting of all bounded linear functionals on $W_0^{s,2}(\Omega)$.
	We split the further proof into a few observations. \newline
	\textbf{Observation 1: $\mathcal{A}$ is well-defined.} 
	Let us show that for any $v \in W_0^{s,2}(\Omega)$, $\mathcal{A}(v)$ is indeed a bounded linear functional and thus belongs to $(W_0^{s,2}(\Omega))^\star$. \newline
	For all $v,\varphi \in W_0^{s,2}(\Omega)$, by (\ref{eq1}), (\ref{PhiLipschitz}) and the Cauchy-Schwarz inequality we have 
	\begin{align*}
	|\langle \mathcal{A}(v),\varphi \rangle| \leq & \lambda^2 \int_{\mathbb{R}^n} \int_{\mathbb{R}^n} \frac{|v(x)-v(y)|}{|x-y|^{n+2s}} |\varphi(x)-\varphi(y)|dydx \\ & + 2 \lambda^2  \int_{\Omega} \int_{\mathbb{R}^n} \frac{|h(x)-h(y)|}{|x-y|^{n+2s}} |\varphi(x)-\varphi(y)|dydx \\
	\leq & \lambda^2  \left ( \int_{\mathbb{R}^n} \int_{\mathbb{R}^n} \frac{|v(x)-v(y)|^2}{|x-y|^{n+2s}} dydx \right)^\frac{1}{2} \left ( \int_{\mathbb{R}^n} \int_{\mathbb{R}^n} \frac{|\varphi(x)-\varphi(y)|^2}{|x-y|^{n+2s}} dydx \right)^\frac{1}{2} \\
	& + 2 \lambda^2 \left ( \int_{\Omega} \int_{\mathbb{R}^n} \frac{|h(x)-h(y)|^2}{|x-y|^{n+2s}} dydx \right)^\frac{1}{2} \left ( \int_{\mathbb{R}^n} \int_{\mathbb{R}^n} \frac{|\varphi(x)-\varphi(y)|^2}{|x-y|^{n+2s}} dydx \right)^\frac{1}{2} \\
	\leq & \lambda^2 (||v||_{W^{s,2}(\mathbb{R}^n)}+2||\nabla^s h||_{L^2(\Omega)}) ||\varphi||_{W^{s,2}(\mathbb{R}^n)}.
	\end{align*}
	Thus, since by (\ref{spaces}) we have $X \subset H^{s,2}(\Omega | \mathbb{R}^n)$ and therefore $\nabla^s h \in L^2(\Omega)$, $\mathcal{A}(v)$ is indeed a bounded linear functional and therefore belongs to $(W_0^{s,2}(\Omega))^\star$. \newline 
	\textbf{Observation 2: $\mathcal{A}$ is monotone.} 
	By (\ref{eq1}) and (\ref{PhiMonotone}), for all $v,w \in W_0^{s,2}(\Omega)$ we have
	\begin{align*}
	& \langle \mathcal{A}_1(v)-\mathcal{A}_1(w),v-w \rangle \\ = & \int_{\Omega} \int_{\mathbb{R}^n} \frac{A(x,y)}{|x-y|^{n+2s}} (\Phi(v(x)+h(x)-v(y)-h(y))-\Phi(w(x)+h(x)-w(y)-h(y))) \\ & \times ((v(x)+h(x)-v(y)-h(y))-(w(x)+h(x)-w(y)-h(y)))dydx \\
	\geq & \lambda^{-2} \int_{\Omega} \int_{\mathbb{R}^n} \frac{((v(x)-v(y))-(w(x)-w(y))^2}{|x-y|^{n+2s}} dydx \geq 0.
	\end{align*}
	By the same reasoning we also have 
	$ \langle \mathcal{A}_2(v)-\mathcal{A}_2(w),v-w \rangle \geq 0$
	and therefore $$ \langle \mathcal{A}(v)-\mathcal{A}(w),v-w \rangle \geq 0,$$
	so that $\mathcal{A}$ is monotone. \newline
	\textbf{Observation 3: $\mathcal{A}$ is weakly continuous.} Let $\{v_j\}_{j=1}^\infty$ be a sequence in $W_0^{s,2}(\Omega)$ that converges to some function $v \in W_0^{s,2}(\Omega)$ in $W^{s,2}(\mathbb{R}^n)$. By (\ref{eq1}), (\ref{PhiLipschitz}) and the Cauchy-Schwarz inequality, for any $\varphi \in W_0^{s,2}(\Omega)$ we obtain
	\begin{align*}
	& |\langle \mathcal{A}(v_j)-\mathcal{A}(v),\varphi \rangle| \\ \leq & \int_{\mathbb{R}^n} \int_{\mathbb{R}^n} \frac{A(x,y)}{|x-y|^{n+2s}} |\Phi(v_j(x)+h(x)-v_j(y)-h(y)) \\
	& -\Phi(v(x)+h(x)-v(y)-h(y))| |\varphi(x)-\varphi(y)|dydx \\
	\leq & \lambda^2 \int_{\mathbb{R}^n} \int_{\mathbb{R}^n} \frac{|(v_j(x)-v_j(y))-(v(x)-v(y))|}{|x-y|^{n+2s}} |\varphi(x)-\varphi(y)|dydx \\
	\leq & \lambda^2 \left (\int_{\mathbb{R}^n} \int_{\mathbb{R}^n} \frac{|(v_j(x)-v(x))-(v_j(y)-v(y))|^2}{|x-y|^{n+2s}} dydx \right )^\frac{1}{2} \\ & \times \left (\int_{\mathbb{R}^n} \int_{\mathbb{R}^n} \frac{|\varphi(x)-\varphi(y)|^2}{|x-y|^{n+2s}} dydx \right )^\frac{1}{2} \\
	\leq & \lambda^2 ||v_j-v||_{W^{s,2}(\mathbb{R}^n)} ||\varphi||_{W^{s,2}(\mathbb{R}^n)} \xrightarrow{j \to \infty} 0.
	\end{align*}
	Therefore, we obtain
	$$ \lim_{j \to \infty} \langle \mathcal{A}(v_j)-\mathcal{A}(v),\varphi \rangle = 0,$$
	which means that $\mathcal{A}$ is weakly continuous. \newline
	\textbf{Observation 4: $\mathcal{A}$ is coercive.} By (\ref{eq1}), (\ref{PhiMonotone}) and (\ref{PhiLipschitz}), for any $v \in W_0^{s,2}(\Omega)$ we have
	\begin{align*}
	& \langle \mathcal{A}_1(v),v \rangle = \langle \mathcal{A}_1(v),v+h \rangle - \langle \mathcal{A}_1(v),h \rangle \\
	= & \int_{\Omega} \int_{\mathbb{R}^n} \frac{A(x,y)}{|x-y|^{n+2s}} \Phi(v(x)+h(x)-v(y)-h(y)) (v(x)+h(x)-v(y)-h(y))dydx \\
	& - \int_{\Omega} \int_{\mathbb{R}^n} \frac{A(x,y)}{|x-y|^{n+2s}} \Phi(v(x)+h(x)-v(y)-h(y)) (h(x)-h(y))dydx \\
	\geq & \lambda^{-2} \bigg( \underbrace{\frac{1}{2} \int_{\mathbb{R}^n} \int_{\mathbb{R}^n} \frac{(v(x)-v(y))^2}{|x-y|^{n+2s}} dydx}_{=:J_1} \\ & \underbrace{ -\int_{\Omega} \int_{\mathbb{R}^n} \frac{|v(x)-v(y)||h(x)-h(y)|}{|x-y|^{n+2s}} dydx}_{=:J_2} - \int_{\Omega} \int_{\mathbb{R}^n} \frac{(h(x)-h(y))^2}{|x-y|^{n+2s}} dydx \bigg ).
	\end{align*}
	By using Lemma \ref{Friedrichsx}, we estimate $J_1$ further from below as follows
	\begin{align*}
	J_1 \geq & \frac{C_1^{-1}|\Omega|^{-\frac{2s}{n}}}{4} ||v||_{L^2(\mathbb{R}^n)}^2 +  \frac{1}{4} \int_{\mathbb{R}^n} \int_{\mathbb{R}^n} \frac{(v(x)-v(y))^2}{|x-y|^{n+2s}} dydx \\ \geq & c ||v||_{W^{s,2}(\mathbb{R}^n)}^2,
	\end{align*}
	where $C_1=C_1(n,s)>0$ is given by Lemma \ref{Friedrichsx} and $c=c(n,s,\lambda,|\Omega|)>0$. 
	By the Cauchy-Schwarz inequality, for $J_2$ we have
	$$ J_2 \geq - \int_{\Omega} \int_{\mathbb{R}^n} \frac{|h(x)-h(y)||v(x)-v(y)|}{|x-y|^{n+2s}} dydx \geq - ||\nabla^s h||_{L^2(\Omega)} ||v||_{W^{s,2}(\mathbb{R}^n)}.$$
	Since by a similar reasoning as above we have
	\begin{align*}
	& \langle \mathcal{A}_2(v),v \rangle = \langle \mathcal{A}_2(v),v+h \rangle - \langle \mathcal{A}_2(v),h \rangle \\
	\geq & \lambda^{-2} \bigg( \int_{\mathbb{R}^n \setminus \Omega} \int_{\Omega} \frac{(v(x)-v(y))^2}{|x-y|^{n+2s}} dydx \\ & -\int_{\mathbb{R}^n \setminus \Omega} \int_{\Omega} \frac{|v(x)-v(y)||h(x)-h(y)|}{|x-y|^{n+2s}} dydx - \int_{\mathbb{R}^n \setminus \Omega} \int_{\Omega} \frac{(h(x)-h(y))^2}{|x-y|^{n+2s}} dydx \bigg ) \\
	\geq & \lambda^{-2} \bigg( -\int_{\Omega} \int_{\mathbb{R}^n} \frac{|v(x)-v(y)||h(x)-h(y)|}{|x-y|^{n+2s}} dydx - \int_{\Omega} \int_{\mathbb{R}^n} \frac{(h(x)-h(y))^2}{|x-y|^{n+2s}} dydx \bigg ),
	\end{align*}
	by combining the last four displays we obtain
	$$ \langle \mathcal{A}(v),v \rangle \geq \lambda^{-2} \left ( c ||v||_{W^{s,2}(\mathbb{R}^n)}^2 - 2 ||\nabla^s h||_{L^2(\Omega)} ||v||_{W^{s,2}(\mathbb{R}^n)} - 2||\nabla^s h||_{L^2(\Omega)}^2 \right ).$$
	Therefore, we conclude that
	\begin{align*}
	& \frac{\langle \mathcal{A}(v),v \rangle}{||v||_{W^{s,2}(\mathbb{R}^n)}} \\ \geq & \lambda^{-2} \left (c ||v||_{W^{s,2}(\mathbb{R}^n)} - 2 ||\nabla^s h||_{L^2(\Omega)} - 2 \frac{||\nabla^s h||_{L^2(\Omega)}^2}{||v||_{W^{s,2}(\mathbb{R}^n)}} \right ) \xrightarrow{||v||_{W^{s,2}(\mathbb{R}^n)} \to + \infty} + \infty, 
	\end{align*}
	which proves that $\mathcal{A}$ is coercive. \newline
	Therefore, since by Remark \ref{Hilbert} $W_0^{s,2}(\Omega)$ is a separable Hilbert space and thus in particular a separable reflexive Banach space, and by the above observations $\mathcal{A}$ is monotone, weakly continuous and coercive, by the standard theory of monotone operators (see e.g. \cite[Corollary 2.2]{Mon}), the operator $\mathcal{A}$ is surjective. Therefore, it remains to prove that the linear functional $$\varphi \mapsto (f,\varphi)_{L^2(\Omega)}, \quad \varphi \in W_0^{s,2}(\Omega)$$ belongs to $(W_0^{s,2}(\Omega))^\star$. Indeed, by H\"older's inequality and the fractional Sobolev inequality (cf. \cite[Theorem 6.5]{Hitch}), for any $\varphi \in W_0^{s,2}(\Omega)$ we have
	\begin{align*}
	(f,\varphi)_{L^2(\Omega)} \leq & \left ( \int_{\Omega} |f(x)|^\frac{2n}{n+2s}dx \right )^\frac{n+2s}{2n} \left (\int_{\mathbb{R}^n} |\varphi(x)|^\frac{2n}{n-2s} dx \right )^\frac{n-2s}{2n} \\
	\leq & C_2 ||f||_{L^\frac{2n}{n+2s}(\Omega)} ||\varphi||_{W^{s,2}(\mathbb{R}^n)},
	\end{align*}
	where $C_2=C_2(n,s)>0$. Therefore, the above functional is indeed bounded and thus belongs to $(W_0^{s,2}(\Omega))^\star$. Hence, by the surjectivity of $\mathcal{A}$ there exists some $v \in W_0^{s,2}(\Omega)$ such that $\langle \mathcal{A}(v),\varphi \rangle = (f,\varphi)_{L^2(\Omega)}$ for all $\varphi \in W_0^{s,2}(\Omega)$. Since by (\ref{spaces}) we have $W_0^{s,2}(\Omega) \subset W^{s,2}(\mathbb{R}^n) \subset X$, we in particular have $v \in X$. Since also $h \in X$ and $X$ is a vector space, the function $u:=v+h$ also belongs to $X$ and satisfies
	\begin{align*}
	(f,\varphi)_{L^2(\Omega)} = & \int_{\Omega} \int_{\mathbb{R}^n} \frac{A(x,y)}{|x-y|^{n+2s}} \Phi(u(x)-u(y))(\varphi(x)-\varphi(y))dydx \\
	& + \int_{\mathbb{R}^n \setminus \Omega} \int_{\Omega} \frac{A(x,y)}{|x-y|^{n+2s}} \Phi(u(x)-u(y))(\varphi(x)-\varphi(y))dydx \\
	= & \int_{\mathbb{R}^n} \int_{\mathbb{R}^n} \frac{A(x,y)}{|x-y|^{n+2s}} \Phi(u(x)-u(y))(\varphi(x)-\varphi(y))dydx,
	\end{align*}
	for all $\varphi \in W_0^{s,2}(\Omega)$. Here we used that $\varphi$ vanishes outside of $\Omega$ in order to obtain the last equality. Since by construction we also have $u=h$ a.e. in $\mathbb{R}^n \setminus \Omega$, $u$ is a weak solution of the Dirichlet problem (\ref{constcof31}). \newline
	Let us prove that this weak solution is unique. Assume that $u_1,u_2 \in X$ both solve the Dirichlet problem (\ref{constcof31}) weakly, so that we have $$u_1-u_2=h-h=0 \text{ in } \mathbb{R}^n \setminus \Omega.$$ Since moreover by (\ref{spaces}) the function $u_1-u_2$ belongs to $H^{s,2}(\Omega | \mathbb{R}^n)$, we clearly have $u_1-u_2 \in W_0^{s,2}(\Omega)$. Therefore, we can use $u_1-u_2$ as a test function in (\ref{constcof31}) for both $u_1$ and $u_2$, so that by subtracting the resulting equalities, along with (\ref{eq1}), (\ref{PhiMonotone}) and Lemma \ref{Friedrichsx} we obtain
	\begin{align*}
	0 = & \int_{\mathbb{R}^n} \int_{\mathbb{R}^n} \frac{A(x,y)}{|x-y|^{n+2s}} (\Phi(u_1(x)-u_1(y))-\Phi(u_2(x)-u_2(y))) \\ & \times ((u_1(x)-u_2(x))-(u_1(y)-u_2(y)))dydx \\
	\geq & \lambda^{-2} \int_{\mathbb{R}^n} \int_{\mathbb{R}^n} \frac{((u_1(x)-u_1(y))-(u_2(x)-u_2(y))^2}{|x-y|^{n+2s}} dydx \\
	\geq & \lambda^{-2} C_1^{-1} |\Omega|^{-\frac{2s}{n}} ||u_1-u_2||_{L^2(\mathbb{R}^n)}^2 \geq 0.
	\end{align*}
	This implies that $||u_1-u_2||_{L^2(\mathbb{R}^n)}=0$ and therefore $u_1 = u_2$ a.e., so that there is exactly one weak solution to the Dirichlet problem (\ref{constcof31}) that belongs to $X$.
\end{proof}

\section{Higher integrability of $\nabla^s u$}
For the rest of this paper, we fix some kernel coefficient $A \in \mathcal{L}_0(\lambda)$ and some function $\Phi:\mathbb{R} \to \mathbb{R}$ satisfying $\Phi(0)=0$, (\ref{PhiLipschitz}) and (\ref{PhiMonotone}). Moreover, we fix some $f \in L^2(B_6)$ and a measurable symmetric function $g:\mathbb{R}^n \times \mathbb{R}^n \to \mathbb{R}$ with $\nabla^s g \in L^2(B_6)$. In addition, for notational clarity we define
$$ Lg:= p.v. \int_{\mathbb{R}^n} \frac{g(x,y)}{|x-y|^{n+2s}} dy,$$
so that the function $F$ defined in (\ref{F}) has the form $F=Lg+f$. \newline
A crucial tool for the proof of the higher integrability of $\nabla^s u$ is given by the following approximation lemma, which shows that any weak solution $u$ of the equation (\ref{nonlocaleq}) is in some sense locally close to a weak solution of a corresponding homogeneous equation that satisfies the H\"older estimate from Proposition \ref{modC2sreg}.
\begin{lemma} \label{apppl}
	Let $M>0$ and assume that $A$ satisfies the condition (\ref{contkernel}) in $B_5$. Then for any $\varepsilon_0 \in (0,1)$, there exists some $\delta = \delta (\varepsilon_0,n,s,\lambda,M) >0$, such that for any weak solution $u \in H^{s,2}(B_5 | \mathbb{R}^n)$ of the equation 
	\begin{equation} \label{eqap}
	L_A^\Phi u =  L g + f \text{ in } B_5
	\end{equation}
	under the assumptions that $A$ satisfies (\ref{contkernel}) in $B_5$, that
	\begin{equation} \label{conddddd}
	\dashint_{B_5} |\nabla^s u|^2 dx \leq M
	\end{equation}
	and that
	\begin{equation} \label{condddddd}
	\dashint_{B_5} \left ( f^{2}+  |\nabla^s g|^2 \right ) dx \leq M\delta^2,
	\end{equation}
	there exists a weak solution $v \in H^{s,2}(B_5 | \mathbb{R}^n)$ of the equation
	\begin{equation} \label{constcof}
	L_A^\Phi v = 0 \text{ in } B_5
	\end{equation}
	that satisfies
	\begin{equation} \label{L2esti}
	||\nabla^s(u-v)||_{L^2(B_5)} \leq \varepsilon_0
	\end{equation}
	and the estimate
	\begin{equation} \label{localAS}
	||\nabla^s v||_{L^\infty(B_2)} \leq N_0 
	\end{equation}
	for some constant $N_0=N_0(n,s,\lambda,M)$.
\end{lemma}

\begin{proof} 
	Fix $\varepsilon_0 \in (0,1)$ and let $\delta>0$ to be chosen.
	Let $v \in H^{s,2}(B_5 | \mathbb{R}^n)$ be the unique weak solution of the problem
	\begin{equation} \label{constcof3}
	\begin{cases} \normalfont
	L_A^\Phi v = 0 & \text{ weakly in } B_5 \\
	v = u & \text{ a.e. in } \mathbb{R}^n \setminus B_5,
	\end{cases}
	\end{equation}
	note that $v$ exists by Proposition \ref{Dirichlet}. In view of (\ref{PhiMonotone}), (\ref{PhiLipschitz}), (\ref{eq1}) and using $w:=u-v \in W_0^{s,2}(B_5)$ as a test function in (\ref{constcof3}) and (\ref{eqap}), we obtain
	\begin{align*}
	||\nabla^s w||_{L^2(B_5)}^2 \leq & \lambda \int_{\mathbb{R}^n} \int_{\mathbb{R}^n} A(x,y) \frac{((u(x)-u(y))-(v(x)-v(y)))^2}{|x-y|^{n+2s}}dydx \\
	\leq & \lambda^2 \bigg ( \int_{\mathbb{R}^n} \int_{\mathbb{R}^n} A(x,y) \frac{\Phi(u(x)-u(y))(w(x)-w(y))}{|x-y|^{n+2s}}dydx \\
	& - \underbrace{\int_{\mathbb{R}^n} \int_{\mathbb{R}^n} A(x,y) \frac{ \Phi(u(x)-u(y))(w(x)-w(y))}{|x-y|^{n+2s}}dydx}_{=0} \bigg ) \\
	= & \lambda^2 \bigg ( \underbrace{\int_{\mathbb{R}^n} \int_{\mathbb{R}^n} \frac{g(x,y)(w(x)-w(y))}{|x-y|^{n+2s}}dydx}_{=:I_1} \\
	& + \underbrace{\int_{B_5} f(x)w(x)dx}_{=:I_2} \bigg ).
	\end{align*}
	By the Cauchy-Schwarz inequality and taking into account that $w=0$ in $\mathbb{R}^n \setminus B_5$, for $I_1$ we have
	\begin{align*}
	I_1 \leq & 2 \int_{B_5} \int_{\mathbb{R}^n} \frac{|g(x,y)||w(x)-w(y)|}{|x-y|^{n+2s}}dydx \\
	\leq & 2 ||\nabla^s g||_{L^2(B_5)} ||\nabla^s w||_{L^2(B_5)},
	\end{align*}
	while by additionally using Lemma \ref{Friedrichsx}, we deduce
	\begin{align*}
	I_2 \leq ||f||_{L^2(B_5)} ||w||_{L^2(B_5)} \leq C_1 ||f||_{L^2(B_5)} ||\nabla^s w||_{L^2(B_5)}.
	\end{align*}
	where $C_1=C_1(n,s)>0$. Therefore, by combining the last three displays we arrive at
	\begin{equation} \label{L2first}
	\begin{aligned}
	||\nabla^s (u-v)||_{L^2(B_5)}^2 \leq & C_2 \left ( ||\nabla^s g||_{L^2(B_5)}^2 + ||\nabla^s f||_{L^2(B_5)}^2 \right ) \\ \leq & 2C_2 |B_5|M \delta^2 \leq \varepsilon_0^2 ,
	\end{aligned}
	\end{equation}
	where the last inequality follows by choosing $\delta$ sufficiently small and $C_2=C_2(n,s,\lambda)>0$.
	This completes the proof of $(\ref{L2esti})$. \newline
	Let us now proof the estimate $(\ref{localAS})$. 
	For almost every $x \in B_2$, by Proposition \ref{taillinf} we have
	$$ \int_{\mathbb{R}^n \setminus B_3} \frac{(v(x)-v(y))^2}{|x-y|^{n+2s}}dy \leq C_3 \int_{B_3} \int_{\mathbb{R}^n} \frac{(v(z)-v(y))^2}{|x-y|^{n+2s}}dydz,$$
	where $C_3=C_3(n,s,\lambda)$.
	Now choose $\gamma>0$ small enough such that $\gamma < s$ and $s + \gamma <1$.
	In view of the assumption that $A$ satisfies (\ref{contkernel}) in $B_5$, by Proposition \ref{modC2sreg} we have 
	$$[v]_{C^{s+\gamma}(B_3)} \leq C_{4} ||\nabla^s v||_{L^2(B_5)} $$
	for some constant $C_{4}=C_{4}(n,s,\lambda,\gamma)$.
	Thus, for almost every $x \in B_2$ we obtain 
	\begin{align*}
	\int_{B_3} \frac{(v(x)-v(y))^2}{|x-y|^{n+2s}}dy & \leq [v]_{C^{s+\gamma}(B_3)}^2 \int_{B_3} \frac{dy}{|x-y|^{n-2\gamma}} \\
	& = C_{5}  [v]_{C^{s+\gamma}(B_3)}^2 \leq C_6 \int_{B_5} \int_{\mathbb{R}^n} \frac{(v(z)-v(y))^2}{|x-y|^{n+2s}}dydz,
	\end{align*}
	where $C_{5}=C_{5}(n,\gamma)<\infty$ and $C_6=C_6(n,s,\lambda,\gamma)>0$.
	By combining the above estimates, along with (\ref{L2first}) and (\ref{conddddd}) we conclude that for almost every $x \in B_2$ we have
	\begin{align*}
	(\nabla^s v)^2 (x) & = \int_{\mathbb{R}^n \setminus B_3} \frac{(v(x)-v(y))^2}{|x-y|^{n+2s}}dy + \int_{B_3} \frac{(v(x)-v(y))^2}{|x-y|^{n+2s}}dy\\
	& \leq C_7 \int_{B_5} \int_{\mathbb{R}^n} \frac{(v(z)-v(y))^2}{|x-y|^{n+2s}}dydz \\
	& \leq 2C_7 \left (\int_{B_5} \int_{\mathbb{R}^n} \frac{(u(z)-u(y))^2}{|x-y|^{n+2s}}dydz + \int_{B_5} \int_{\mathbb{R}^n} \frac{(w(z)-w(y))^2}{|x-y|^{n+2s}}dydz \right ) \\
	& \leq 2C_7 \left (\int_{B_5} \int_{\mathbb{R}^n} \frac{(u(z)-u(y))^2}{|x-y|^{n+2s}}dydz + \varepsilon_0^2 \right ) \leq 2C_7 (|B_5|M+1),
	\end{align*}
	where $C_7=C_7(n,s,\gamma,\lambda)$. Therefore, (\ref{localAS}) holds with $N_0=(2C_7 (|B_5|M+1))^{\frac{1}{2}}$.
\end{proof}

The following result is an application of the above approximation lemma and roughly speaking shows that if the maximal functions of $\nabla^s u$, $\nabla^s f$ and $\nabla^s g$ are small enough in some point, then the set where the maximal function of $\nabla^s u$ is large has to be very small.
\begin{lemma} \label{mfuse}
	There is a constant $N_1=N_1(n,s,\lambda) > 1$, such that the following is true. If $A$ satisfies the condition (\ref{contkernel}) in $B_6$, then for any $\varepsilon > 0$ there exists some $\delta = \delta(\varepsilon,n,s,\lambda) > 0$, 
	such that for any $z \in B_1$, any $r \in (0,1]$ and any weak solution $ u \in H^{s,2}(B_{5r}(z) | \mathbb{R}^n)$ 
	of the equation 
	$$ L_A^\Phi u =  Lg + f \text { in } B_{5r}(z)$$
	with
	\begin{align*}
	& \left \{ x \in B_r(z) \mid \mathcal{M}_{B_6} (|\nabla^s u|^2)(x) \leq 1 \right \} \\ & \cap \left \{ x \in B_r(z) \mid \mathcal{M}_{B_6} \left (|f|^{2} +  |\nabla^s g|^2 \right )(x) \leq \delta^2 \right \} \neq \emptyset,
	\end{align*}
	we have  
	\begin{equation} \label{ccll}
	\left | \left \{ x \in B_r(z) \mid \mathcal{M}_{B_6} (|\nabla^s u|^2)(x) > N_1^2 \right \} \right | < \varepsilon |B_r|.
	\end{equation}
\end{lemma}

\begin{proof}
	Let $\varepsilon_0 \in (0,1)$ and $M>0$ to be chosen and consider the corresponding $\delta = \delta(\varepsilon_0,n,s,\lambda,M) > 0$ given by Lemma $\ref{apppl}$. 
	Fix $r \in (0,1]$ and $z \in \mathbb{R}^n$. Define 
	\begin{align*}
	\widetilde A (x,y):= A (rx+z,ry+z), \quad \widetilde u (x) := r^{-s} u(rx+z),\\
	\widetilde g (x,y) := r^{-s} g(rx+z,ry+z), \quad \widetilde f (x) := r^s f(rx+z)
	\end{align*}
	and note that under the above assumptions 
	$\widetilde A$ belongs to the class $\mathcal{L}_0(\lambda)$ and satisfies the condition (\ref{contkernel}) in $B_5$, and that $\widetilde u \in H^{s,2}(B_5 | \mathbb{R}^n)$ satisfies 
	$$ L_{\widetilde A}^\Phi \widetilde u =  L \widetilde g + \widetilde f \text { weakly in } B_5. $$
	Therefore, by Lemma \ref{apppl} there exists a weak solution $\widetilde v \in H^{s,2}(B_5 | \mathbb{R}^n)$ of 
	$$ L_{\widetilde A}\widetilde v = 0 \text{ in } B_5 $$
	such that 
	\begin{equation} \label{apss}
	\int_{B_2} |\nabla^s(\widetilde u -\widetilde v)|^2 dx \leq \varepsilon_0^2,
	\end{equation}
	provided that the conditions $(\ref{conddddd})$ and $(\ref{condddddd})$ are satisfied.
	By assumption, there exists a point $x \in B_r(z)$ such that 
	$$\mathcal{M}_{B_6} (|\nabla^s u|^2)(x) \leq 1, \quad \mathcal{M}_{B_6} \left (|f|^{2} +  |\nabla^s g|^2 \right )(x) \leq \delta^2.$$
	By the scaling and translation invariance of the Hardy-Littlewood maximal function (Lemma $\ref{scale}$), for $x_0:=\frac{x-z}{r} \in B_1$ we thus have
	$$ \mathcal{M}_{B_{6/r}(-z)} (|\nabla^s \widetilde u|^2)(x_0) = \mathcal{M}_{B_6} (|\nabla^s u|^2)(x) \leq 1$$
	and
	$$ \mathcal{M}_{B_{6/r}(-z)} \left (|\widetilde f|^{2} +  |\nabla^s \widetilde g|^2 \right )(x_0) = \mathcal{M}_{B_6} \left (r^{2s}|f|^{2} +  |\nabla^s g|^2 \right )(x) \leq \delta^2.$$
	Therefore, for any $\rho>0$ we have 
	\begin{equation} \label{lum}
	\dashint_{B_\rho(x_0)} |\nabla^s \widetilde u|^2 dx \leq 1, \quad \dashint_{B_\rho(x_0)} \left ( |\widetilde f|^{2}+  |\nabla^s \widetilde g|^2 \right ) dx \leq \delta^2,
	\end{equation}
	where the values of $\nabla^s \widetilde u$, $\nabla^s \widetilde g$ and $\widetilde f$ outside of $B_{6/r}(-z)$ are replaced by $0$, which we also do for the rest of the proof.
	Since $B_5 \subset B_6(x_0)$, by $(\ref{lum})$ we have 
	$$
	\dashint_{B_5} |\nabla^s \widetilde u|^2 dx \leq \frac{|B_6|}{|B_5|} \text{ } \dashint_{B_6(x_0)} |\nabla^s \widetilde u|^2 dx \leq \left (\frac{6}{5} \right )^n
	$$
	and
	$$
	\dashint_{B_5} \left ( |\widetilde f|^2 +  |\nabla^s \widetilde g|^2 \right ) dx \leq \frac{|B_6|}{|B_5|} \text{ } \dashint_{B_6(x_0)} \left ( |\widetilde f|^2 +  |\nabla^s \widetilde g|^2 \right ) dx \leq \left (\frac{6}{5} \right )^n \delta^2.
	$$
	Since also $B_5 \subset B_{6/r}(-z)$, we obtain that $\widetilde u$, $\widetilde g$ and $\widetilde f$ satisfy 
	the conditions $(\ref{conddddd})$ and $(\ref{condddddd})$ with $M=\left (\frac{6}{5} \right )^n$. Therefore, (\ref{apss}) is satisfied by $\widetilde u$ and the corresponding approximate solution $\widetilde v$. Considering the function $v \in H^{s,2}(B_6|\mathbb{R}^n)$ given by $v(x):=r^{s} \widetilde v \left (\frac{x-z}{r} \right)$ and rescaling back yields 
	\begin{equation} \label{apss91}
	\int_{B_{2r}(y)} |\nabla^s(u -v)|^2dx = r^n \int_{B_2} |\nabla^s(\widetilde u -\widetilde v)|^2dx \leq \varepsilon_0^2 r^n.
	\end{equation}
	By Lemma $\ref{apppl}$, there is a constant $N_0=N_0(n,s, \lambda) >0$ such that 
	\begin{equation} \label{loclinf}
	||\nabla^s \widetilde v||_{L^\infty(B_2)}^2 \leq N_0^2 . 
	\end{equation}
	Next, we define $N_1 := (\max \{ 4 N_0^2, 3^n \})^{1/2} > 1$ and claim that 
	\begin{equation} \label{inclusion}
	\begin{aligned}
	& \left \{ x \in B_1 \mid \mathcal{M}_{B_{6/r}(-z)} ( |\nabla^s \widetilde u|^2 )(x) > N_1^2 \right \} \\ \subset & \left \{ x \in B_1 \mid \mathcal{M}_{B_2} ( |\nabla^s(\widetilde u -\widetilde v)|^2 )(x) > N_0^2 \right \}. 
	\end{aligned}
	\end{equation}
	In order to see this, assume that 
	\begin{equation} \label{menge}
	x_1 \in \left \{ x \in B_1 \mid \mathcal{M}_{B_2} ( |\nabla^s(\widetilde u -\widetilde v)|^2 ) (x) \leq N_0^2 \right \}. 
	\end{equation}
	For $ \rho < 1$, we have $B_\rho (x_1) \subset B_1(x_1) \subset B_2$, so that together with $(\ref{menge})$ and $(\ref{loclinf})$ we deduce 
	\begin{align*}
	\dashint_{B_\rho (x_1)} |\nabla^s \widetilde u|^2 dx & \leq 2 \text{ } \dashint_{B_\rho(x_1)} \left ( |\nabla^s (\widetilde u -\widetilde v)|^2 + |\nabla^s \widetilde v|^2 \right )dx \\
	& \leq 2 \text{ } \dashint_{B_\rho(x_1)} |\nabla^s (\widetilde u -\widetilde v)|^2 dx + 2 \text{ } ||\nabla^s \widetilde v||_{L^\infty(B_\rho(x_1))}^2 \\
	& \leq 2 \text{ } \mathcal{M}_{B_2} (|\nabla^s (\widetilde u -\widetilde v)|^2) (x_1) + 2 \text{ } ||\nabla^s \widetilde v||_{L^\infty(B_2)}^2 \leq 4 N_0^2. 
	\end{align*}
	On the other hand, for $\rho \geq 1$ we have $ B_\rho (x_1) \subset B_{3 \rho}(x_0)$, so that $(\ref{lum})$ implies 
	$$
	\dashint_{B_\rho(x_1)} |\nabla^s \widetilde u|^2 dx \leq \frac{|B_{3 \rho}|}{|B_\rho|} \dashint_{B_{3 \rho} (x_0)} |\nabla^s \widetilde u|^2 dx \leq 3^n.
	$$
	Thus, we have $$ x_1 \in \left \{ x \in B_1 \mid \mathcal{M}_{B_{6/r}(-z)}( |\nabla^s \widetilde u|^2) (x) \leq N_1^2 \right \} ,$$ 
	which implies $(\ref{inclusion})$. In view of Lemma \ref{scale}, (\ref{inclusion}) is equivalent to
	\begin{equation} \label{inclusion91}
	\begin{aligned}
	& \left \{ x \in B_r(z) \mid \mathcal{M}_{B_6} ( |\nabla^s u|^2 )(x) > N_1^2 \right \} \\ \subset & \left \{ x \in B_r(z) \mid \mathcal{M}_{B_{2r}(z)} ( |\nabla^s (u -v)|^2 )(x) > N_0^2 \right \}.
	\end{aligned}
	\end{equation}
	For any $\varepsilon > 0$, using $(\ref{inclusion91})$, the weak $1$-$1$ estimate from Proposition \ref{Maxfun} and $(\ref{apss91})$, we conclude that there exists some constant $C_1=C_1(n)>0$ such that 
	\begin{align*}
	& \left | \left \{ x \in B_r(z) \mid \mathcal{M}_{B_6} ( |\nabla^s u|^2)(x) >N_1^2 \right \} \right | \\
	\leq & \left |  \left \{ x \in B_r(z) \mid \mathcal{M}_{B_{2r}(z)} ( |\nabla^s (u -v)|^2 )(x) > N_0^2 \right \} \right | \\
	\leq & \frac{C_1}{N_0^2} \int_{B_{2r}(z)} |\nabla^s(u -v)|^2 dx \\
	\leq & \frac{C_1}{N_0^2} \varepsilon_0^2 r^n = \frac{C_2}{N_0^2} \varepsilon_0^2 |B_r| < \varepsilon |B_r|,
	\end{align*}
	where $C_2=C_2(n)>0$ and the last inequality is obtained by choosing $\varepsilon_0$ and thus also $\delta$ sufficiently small. This finishes our proof.
\end{proof}

\begin{corollary} \label{applic}
	There is a constant $N_1=N_1(n,s,\lambda) > 1$, such that the following is true. If $A$ satisfies the condition (\ref{contkernel}) in $B_6$, then for any $\varepsilon > 0$ there exists some $\delta = \delta(\varepsilon,n,s,\lambda) > 0$, 
	such that for any $z \in B_1$, any $r \in (0,1)$ and any weak solution $ u \in H^{s,2}(B_6 | \mathbb{R}^n)$ 
	of the equation 
	$$ L_A^\Phi u =  L g + f \text { in } B_6$$
	with
	\begin{equation} \label{Lvv}
	\left | \left \{ x \in B_r(z) \mid \mathcal{M}_{B_6} (|\nabla^s u|^2)(x) > N_1^2 \right \} \cap B_1 \right | \geq \varepsilon |B_r|,
	\end{equation}
	we have
	\begin{equation} \label{Lvv1}
	\begin{aligned}
	B_r(z) \cap B_1 \subset & \left \{ x \in B_1 \mid \mathcal{M}_{B_6} (|\nabla^s u|^2)(x) > 1 \right \} \\
	& \cup \left \{ x \in B_1 \mid \mathcal{M}_{B_6} \left (|f|^{2} +  |\nabla^s g|^2 \right )(x) > \delta^2 \right \}.
	\end{aligned}
	\end{equation}
\end{corollary}

\begin{proof}
	Let $N_1=N_1(n,s,\lambda)>1$ be given by Lemma $\ref{mfuse}$.
	Fix $\varepsilon > 0$, $r \in (0,1)$, $z \in \mathbb{R}^n$ and consider the corresponding $\delta= \delta (\varepsilon,n,s,\lambda)>0$ given by Lemma $\ref{mfuse}$.
	We now argue by contradiction.
	Assume that $(\ref{Lvv})$ is satisfied but that $(\ref{Lvv1})$ is false, so that there exists some $x_0 \in B_r(z) \cap B_1$ such that 
	\begin{align*} 
	x_0 \in B_r(z) & \cap \left \{ x \in B_1 \mid \mathcal{M}_{B_{6}} (|\nabla^s u|^2)(x) \leq 1 \right \} \\
	& \cap \left \{ x \in B_1 \mid \mathcal{M}_{B_6} \left (|f|^{2} +  |\nabla^s g|^2 \right )(x) \leq \delta^2 \right \} \\
	\subset &  \left \{ x \in B_r(z) \mid \mathcal{M}_{B_6} (|\nabla^s u|^2)(x) \leq 1 \right \} \\
	& \cap \left \{ x \in B_r(z) \mid \mathcal{M}_{B_6} \left (|f|^{2} +  |\nabla^s g|^2 \right )(x) \leq \delta^2 \right \} .
	\end{align*}
	Since in addition we have $B_{5r}(z) \subset B_6$, by Lemma $\ref{mfuse}$ we arrive at
	\begin{align*}
	& \left | \left \{ x \in B_r(z) \mid \mathcal{M}_{B_6} (|\nabla^s u|^2)(x) > N_1^2 \right \} \cap B_1 \right | \\
	\leq & \left | \left \{ x \in B_r(z) \mid \mathcal{M}_{B_6} (|\nabla^s u|^2)(x) > N_1^2 \right \} \right | < \varepsilon |B_r|,
	\end{align*}
	which contradicts $(\ref{Lvv})$.
\end{proof}

The following decay of level sets will be the main key to proving the higher integrability of $\nabla^s u$.
\begin{lemma} \label{aplvi}
	Let $N_1=N_1(n,s,\lambda) > 1$ be given by Corollary $\ref{applic}$ and assume that $A$ satisfies the condition (\ref{contkernel}) in $B_6$.
	Moreover, let $k \in \mathbb{N}$, $\varepsilon \in (0,1)$, set $\varepsilon_1 := 10^n \varepsilon$ 
	and consider the corresponding $\delta = \delta(\varepsilon,n,s,\lambda)>0$ given by Corollary $\ref{applic}$.
	Then for any weak solution $ u \in H^{s,2}(B_6 | \mathbb{R}^n)$ 
	of the equation 
	$$ L_A^\Phi u =  L g + f \text { in } B_6$$
	with
	\begin{equation} \label{air}
	\left | \left \{ x \in B_1 \mid \mathcal{M}_{B_6} (|\nabla^s u|^2)(x) > N_1^2 \right \} \right | < \varepsilon |B_1| ,
	\end{equation}
	we have 
	\begingroup
	\allowdisplaybreaks
	\begin{align*}
	& \left | \left \{ x \in B_1 \mid \mathcal{M}_{B_6} (|\nabla^s u|^2)(x) > N_1^{2k} \right \} \right | \\
	\leq & \sum_{j=1}^k \varepsilon_1^j \left | \left \{ x \in B_1 \mid \mathcal{M}_{B_6} \left (|f|^{2} +  |\nabla^s g|^2 \right )(x) > \delta^2 N_1^{2(k-j)} \right \} \right | \\ 
	& + \varepsilon_1^k \left | \left \{ x \in B_1 \mid \mathcal{M}_{B_6} (|\nabla^s u|^2)(x) > 1 \right \} \right | .
	\end{align*}
	\endgroup
\end{lemma}

\begin{proof}
	We proof this Lemma by induction on $k$. 
	In view of $(\ref{air})$ and Corollary $\ref{applic}$, the case $k=1$ is a direct consequence of Lemma $\ref{modVitali}$
	applied to the sets 
	$$ E := \left \{ x \in B_1 \mid \mathcal{M}_{B_6} (|\nabla^s u|^2)(x) > N_1^2 \right \}  $$
	and
	\begin{align*}
	F:= \left \{ x \in B_1 \mid \mathcal{M}_{B_6} (|\nabla^s u|^2)(x) > 1 \right \} \cup \left \{ x \in B_1 \mid \mathcal{M}_{B_6} \left (|f|^{2} +  |\nabla^s g|^2 \right )(x) > \delta^2 \right \} .
	\end{align*}
	Next, assume that the conclusion is valid for some $k \in \mathbb{N}$.
	Define $\widehat \Phi(t):=\Phi(N_1 t)/N_1$, $\widehat u := u/N_1$, $\widehat g := g/N_1$ and $\widehat f := f/N_1$. Then $\widehat \Phi$ clearly satisfies the conditions (\ref{PhiLipschitz}) and (\ref{PhiMonotone}) with respect to $\lambda$ and we have  
	$$ L_A^{\widehat \Phi} \widehat u =  L \widehat g + \widehat f \text { weakly in } B_6.$$
	Moreover, since $N_1 > 1$, we have 
	\begin{align*}
	& \left | \left \{ x \in B_1 \mid \mathcal{M}_{B_6} (|\nabla^s \widehat u|^2)(x) > N_1^2 \right \} \right | \\
	= & \left | \left \{ x \in B_1 \mid \mathcal{M}_{B_6} (|\nabla^s u|^2)(x) > N_1^4 \right \} \right | \\
	\leq & \left | \left \{ x \in B_1 \mid \mathcal{M}_{B_6} (|\nabla^s u|^2)(x) > N_1^2 \right \} \right |< \varepsilon |B_1|.
	\end{align*}
	Thus, using the induction assumption yields 
	\begin{align*}
	& \left | \left \{ x \in B_1 \mid \mathcal{M}_{B_6} (|\nabla^s u|^2)(x) > N_1^{2(k+1)} \right \} \right | \\
	= & \left | \left \{ x \in B_1 \mid \mathcal{M}_{B_6} (|\nabla^s \widehat u|^2)(x) > N_1^{2k} \right \} \right | \\
	\leq & \sum_{j=1}^k \varepsilon_1^j \left | \left \{ x \in B_1 \mid \mathcal{M}_{B_6} \left (|\widehat f|^{2} +  |\nabla^s \widehat g|^2 \right )(x) > \delta^2 N_1^{2(k-j)} \right \} \right | \\
	& + \varepsilon_1^k \left | \left \{ x \in B_1 \mid \mathcal{M}_{B_6} (|\nabla^s \widehat u|^2)(x) > 1 \right \} \right | \\
	= & \sum_{j=1}^k \varepsilon_1^j \left | \left \{ x \in B_1 \mid \mathcal{M}_{B_6} \left (|f|^{2} +  |\nabla^s g|^2 \right )(x) > \delta^2 N_1^{2(k+1-j)} \right \} \right | \\
	& + \varepsilon_1^k \left | \left \{ x \in B_1 \mid \mathcal{M}_{B_6} (|\nabla^s u|^2)(x) > N_1^2 \right \} \right |.
	\end{align*}
	Moreover, by using the case $k=1$ we obtain
	\begin{align*}
	= & \sum_{j=1}^k \varepsilon_1^j \left | \left \{ x \in B_1 \mid \mathcal{M}_{B_6} \left (|f|^{2} +  |\nabla^s g|^2 \right )(x) > \delta^2 N_1^{2(k+1-j)} \right \} \right | \\
	& + \varepsilon_1^k \left | \left \{ x \in B_1 \mid \mathcal{M}_{B_6} (|\nabla^s u|^2)(x) > N_1^2 \right \} \right | \\
	\leq & \sum_{j=1}^k \varepsilon_1^j \left | \left \{ x \in B_1 \mid \mathcal{M}_{B_6} \left (|f|^{2} +  |\nabla^s g|^2 \right )(x) > \delta^2 N_1^{2(k+1-j)} \right \} \right | \\
	& + \varepsilon_1^k \Bigg ( \varepsilon_1 \left | \left \{ x \in B_1 \mid \mathcal{M}_{B_6} \left (|f|^{2} +  |\nabla^s g|^2 \right )(x) > \delta^2 \right \} \right | \\
	& + \varepsilon_1 \left | \left \{ x \in B_1 \mid \mathcal{M}_{B_6} (|\nabla^s u|^2)(x) > 1 \right \} \right | \Big ) \\
	= & \sum_{j=1}^{k+1} \varepsilon_1^j \Bigg ( \left | \left \{ x \in B_1 \mid \mathcal{M}_{B_6} \left (|f|^{2} +  |\nabla^s g|^2 \right )(x) > \delta^2 N_1^{2(k+1-j)} \right \} \right | \\
	& + \varepsilon_1^{k+1} \left | \left \{ x \in B_1 \mid \mathcal{M}_{B_6} (|\nabla^s u|^2)(x) > 1 \right \} \right | ,
	\end{align*}
	so that by combining the last two displays we see that the conclusion is valid for $k+1$, which completes the proof.
\end{proof}

We are now set to prove the desired higher integrability of $\nabla^s u$ in the case of balls. The main tools are Lemma \ref{aplvi}, Lemma \ref{Caff} and Proposition \ref{Maxfun}.
\begin{theorem} \label{mainint1}
	Let $2<p< \infty$. Moreover, let $g:\mathbb{R}^n \times \mathbb{R}^n \to \mathbb{R}$ be a measurable symmetric function with $\nabla^s g \in L^p(B_6)$ and $f \in L^p(B_6)$. If $A \in \mathcal{L}_0(\lambda)$ satisfies the condition (\ref{contkernel}) in $B_6$ and if $\Phi$ satisfies and the assumptions (\ref{PhiLipschitz}) and (\ref{PhiMonotone}) with respect to $\lambda$, then for any weak solution $u \in H^{s,2}(B_6 | \mathbb{R}^n)$ 
	of the equation 
	$$ L_A^\Phi u =  L g  + f \text{ in } B_6,$$ 
	we have $\nabla^s u \in L^p(B_1)$. 
	Moreover, there exists a constant $C= C(p,n,s,\lambda) >0$ such that
	\begin{equation} \label{Lpest98}
	||\nabla^s u||_{L^p(B_1)} \leq C \left (||f||_{L^p(B_6)} + ||\nabla^s {g}||_{L^p(B_6)} + ||\nabla^s u||_{L^2(B_6)} \right ). 
	\end{equation}
\end{theorem}

\begin{proof}
	Fix $p>2$ and let $N_1=N_1(n,s,\lambda) > 1$ be given by Lemma $\ref{aplvi}$. Moreover, select $\varepsilon \in (0,1)$ small enough such that 
	\begin{equation} \label{dwn}
	N_1^p 10^n \varepsilon \leq \frac{1}{2}.
	\end{equation}
	Consider also the corresponding $\delta = \delta(\varepsilon,n,s,\lambda)>0$ given by Corollary $\ref{applic}$.
	If $\nabla^s u=0$ a.e$.$ in $B_6$, then the assertion is trivially satisfied, so that we can assume $||\nabla^s u||_{L^2(B_6)} > 0$.
	Next, we let $\gamma >0$ to be chosen independently of $u$, $g$ and $f$ and define
	$$u_\gamma := \frac{u}{M_\gamma}, \text{ } g_\gamma := \frac{g}{M_\gamma} \text{ and } f_\gamma := \frac{f}{M_\gamma},$$ 
	where $M_\gamma:=||\nabla^s u||_{L^2(B_6)}/ \gamma$.
	In addition, we define $\Phi_\gamma(t):=\frac{1}{M_\gamma} \Phi(M_\gamma t)$ and note that $\Phi_\gamma$ satisfies (\ref{PhiLipschitz}) and (\ref{PhiMonotone}) with respect to $\lambda$ and that we have 
	$$ L_A^{\Phi_\gamma} u_\gamma =  L g_\gamma + f_\gamma \text{ weakly in } B_6.$$
	Moreover, we have
	$$ \int_{B_6} |\nabla^s u_\gamma|^2 dx = \gamma^2.$$
	By combining this observation with the weak $1$-$1$ estimate from Proposition \ref{Maxfun}, it follows that there is a constant $C_1=C_1(n)>0$ such that 
	$$\left | \left \{ x \in B_1 \mid \mathcal{M}_{B_6} (|\nabla^s u_\gamma|^2)(x) > N_1^2 \right \} \right | \leq \frac{C_1}{N_1^2} \int_{B_6} |\nabla^s u_\gamma|^2 dx 
	= \frac{C_1 \gamma^2}{N_1^2} < \varepsilon |B_1|, $$
	where the last inequality is obtained by choosing $\gamma$ small enough.
	Therefore, all assumptions made in Lemma $\ref{aplvi}$ are satisfied by $u_\gamma$.
	Furthermore, by Proposition \ref{Maxfun} and Lemma $\ref{Caff}$ applied with $\tau=1$, $\beta=N_1^2$ and with $p$ replaced by $p/2$, we deduce that
	\begin{align*}
	||\nabla^s u_\gamma||^p_{L^p(B_1)} & \leq ||\mathcal{M}_{B_6}(|\nabla^s u_\gamma|^2)||_{L^{p/2}(B_1)}^{p/2} \\
	& \leq C_2 \left ( \sum_{k=1}^\infty N_1^{pk} \left | \left \{ x \in B_1 \mid \mathcal{M}_{B_6} (|\nabla^s u_\gamma|^2)(x) > N_1^{2k} \right \} \right | 
	+ |B_1| \right ),
	\end{align*}
	where $C_2=C_2(n,s,p,\lambda)>0$.
	Setting $\varepsilon_1 := 10^n \varepsilon$, by $(\ref{dwn})$ we see that 
	\begin{equation} \label{Slh}
	\sum_{j=1}^\infty (N_1^p \varepsilon_1)^{j} \leq \sum_{j=1}^\infty \left ( \frac{1}{2} \right )^{j} = 1.
	\end{equation}
	Using Lemma $\ref{aplvi}$, the Cauchy product and $(\ref{Slh})$, we compute 
	\begin{align*} 
	& \sum_{k=1}^\infty N_1^{pk} \left | \left \{ x \in B_1 \mid \mathcal{M}_{B_6} (|\nabla^s u_\gamma|^2)(x) > N_1^{2k} \right \} \right | \\
	\leq & \sum_{k=1}^\infty N_1^{pk} \Bigg ( \sum_{j=1}^k \varepsilon_1^j \left 
	| \left \{ x \in B_1 \mid \mathcal{M}_{B_6} \left (|f_\gamma|^{2}+ |\nabla^s g_\gamma|^2 \right )(x)> \delta^2 N_1^{2(k-j)} \right \} \right | \\
	& + \varepsilon_1^k \left | \left \{ x \in B_1 \mid \mathcal{M}_{B_6} (|\nabla^s u_\gamma|^2)(x) > 1 \right \} \right | \Bigg ) \\
	= & \left ( \sum_{k=0}^\infty N_1^{pk} \left | \left \{ x \in B_1 \mid \mathcal{M}_{B_6} \left (|f_\gamma|^{2}+ |\nabla^s g_\gamma|^2 \right )(x) > \delta^2 N_1^{2k} \right \} \right | \right) 
	\left ( \sum_{j=1}^\infty (N_1^p \varepsilon_1)^{j} \right ) \\
	& + \left ( \sum_{k=1}^\infty (N_1^p \varepsilon_1)^{k} \right ) \left | \left \{ x \in B_1 \mid \mathcal{M}_{B_6} (|\nabla^s u_\gamma|^2)(x) > 1 \right \} \right | \\
	\leq & \sum_{k=1}^\infty N_1^{pk} \left | \left \{ x \in B_1 \mid \mathcal{M}_{B_6} \left (|f_\gamma|^{2}+ |\nabla^s g_\gamma|^2 \right )(x) > \delta^2 N_1^{2k} \right \} \right | + |B_1|.
	\end{align*}
	Next, by combining the previous two displays with Lemma $\ref{Caff}$ applied with with $\tau=\delta^2$, $\beta=N_1^2$ and with $p$ replaced by $p/2$, and also taking into account the strong $p$-$p$ estimates from Proposition \ref{Maxfun}, we deduce that
	\begin{align*}
	& ||\nabla^s u_\gamma||^p_{L^p(B_1)} \\ \leq & C_2 \left ( \sum_{k=1}^\infty N_1^{pk} \left | \left \{ x \in B_1 \mid \mathcal{M}_{B_6} \left (|f_\gamma|^{2} +  |\nabla^s g_\gamma|^2 \right )(x) > \delta^2 N_1^{2k} \right \} \right | + 2|B_1| \right ) \\
	\leq & C_3 \left (||\mathcal{M}_{B_6} \left (|f_\gamma|^{2} +  |\nabla^s g_\gamma|^2 \right )||_{L^{p/2}(B_6)}^{p/2}+1\right ) \\
	\leq & C_4 \left (||f_\gamma||_{L^p(B_6)}^{p} + || \nabla^s g_\gamma||_{L^p(B_6)}^{p}+1 \right ),
	\end{align*}
	where $C_3=C_3(n,s,p,\lambda)>0$ and $C_4=C_4(n,p)>0$.
	It follows that 
	$$ ||\nabla^s u_\gamma||_{L^p(B_1)}
	\leq C_4^{1/p} \left ( ||f_\gamma||_{L^p(B_6)} + || \nabla^s g_\gamma||_{L^p(B_6)} +1 \right ), $$
	so that
	\begin{align*}
	||\nabla^s u||_{L^p(B_1)} & \leq C_4^{1/p} \left ( ||f||_{L^p(B_6)} + || \nabla^s {g}||_{L^p(B_6)} + \frac{||\nabla^s u||_{L^2(B_6)}}{\gamma} \right ) \\
	& \leq C \left ( ||f||_{L^p(B_6)} + || \nabla^s {g}||_{L^p(B_6)} + ||\nabla^s u||_{L^2(B_6)} \right ),
	\end{align*}
	which proves the estimate (\ref{Lpest98}).
\end{proof}

\section{Proof of the main result}
We are now set to prove our main result by using scaling and covering argments.
\begin{proof}[Proof of Theorem \ref{mainint5}]
	Fix $p \in (2,\infty)$. We first assume that $f \in L^p_{loc}(\Omega)$.
	Fix relatively compact bounded open sets $U \Subset V \Subset \Omega$. Moreover, fix a Lipschitz domain $U_\star$ such that $U \Subset U_\star \Subset V$. 
	For any $z \in V$, fix some small enough $r_z \in (0,1)$ such that $B_{6r_z}(z) \Subset V$.
	Define 
	\begin{align*}
	A_z (x,y):= A \left (r_z x+z, r_z y+z \right ), \quad u_z (x) := r^{-s}_z u \left (r_z x +z \right ), \\
	{g}_z (x) := r^{-s}_z g \left (r_z x +z,r_z y +z \right ), \quad f_z (x) := r^s_z f \left (r_z x +z \right )
	\end{align*} 
	and note that for any $z \in V$, 
	$A_z$ belongs to the class $\mathcal{L}_0(\lambda)$ and satisfies the condition (\ref{contkernel}) in $B_6$, and that $u_z \in H^{s,2}(B_6)$ satisfies 
	$$ L_{A_z}^\Phi u_z = L {g}_z + f_z \text { weakly in } B_6.$$ 
	By Theorem $\ref{mainint1}$, we obtain the estimate 
	\begin{align*}
	& ||\nabla^s u||_{L^p \left (B_{r_z}(z) \right)} = r_z^{\frac{n}{p}} ||\nabla^s u_z||_{L^q(B_1)} \\
	\leq & r_z^{\frac{n}{p}} C_1 
	\left (||f_z||_{L^p(B_6)} + || \nabla^s {g}_z||_{L^p(B_6)} + ||\nabla^s u_z||_{L^2(B_6)} \right ) \\
	= & C_1 \left (r^s_z ||f||_{L^p(B_{6r_z}(z))} + || \nabla^s g||_{L^p(B_{6r_z}(z))} + r_z^{\frac{n}{p}-\frac{n}{2}} ||\nabla^s u||_{L^2(B_{6r_z}(z))} \right ) \\
	\leq & C_1 \max \{1, r_z^{\frac{n}{p}-\frac{n}{2}} \} \left (||f||_{L^q(B_{6r_z}(z))} + || \nabla^s g||_{L^p(B_{6r_z}(z))} + ||\nabla^s u||_{L^2(B_{6r_z}(z))} \right ),
	\end{align*}
	where $C_1=C_1(p,n,s,\lambda) >0$.
	Since $\left \{ B_{r_z}(z) \right \}_{z \in \overline U_\star} $ is an open covering of $\overline U_\star$ and $\overline U_\star$ is compact, there is a finite subcover 
	$ \left \{ B_{r_{z_i}}(z_i) \right \}_{i =1}^k$ of $\overline U_\star$ and hence of $U_\star$. 
	Let $\{\phi_i \}_{i=1}^k$ be a partition of unity subordinate to the covering $\left \{ B_{r_{z_i}}(z_i) \right \}_{i =1}^k$ of $\overline U_\star$, that is, the $\phi_i$ are non-negative functions on $\mathbb{R}^n$, we have 
	$ \phi_i \in C_0^\infty(B_{r_{x_i}}(x_i))$ for all $i=1,...,k$, $\sum_{i=1}^k \phi_j \equiv 1$ in an open neighbourhood of $\overline U_\star$ and $\sum_{i=1}^k \phi_j \leq 1$ in $\mathbb{R}^n$.
	Setting $C_2:=C_1 \max \{1, \max_{i=1,...,k} r_{z_i}^{\frac{n}{p}-\frac{n}{2}} \}$ and summing the above estimates over $i=1,...,k$, we conclude 
	\begin{align*} 
	||\nabla^s u||_{L^p(U_\star)} & =||\sum_{i=1}^k |\nabla^s u| \phi_i ||_{L^q(U_\star)} \\
	& \leq \sum_{i=1}^k |||\nabla^s u| \phi_i||_{L^p(B_{r_{z_i}}(z_i))} \\
	& \leq \sum_{i=1}^k ||\nabla^s u||_{L^p(B_{r_{z_i}(z_i) })} \\
	& \leq \sum_{i=1}^k C_2 \left (||f||_{L^p(B_{6r_z}(z))} + || \nabla^s g||_{L^p(B_{6r_z}(z))} + ||\nabla^s u||_{L^2(B_{6r_z}(z))} \right ) \\
	& \leq C_2 k \left (||f||_{L^p(V)} + || \nabla^s g||_{L^p(V)} + ||\nabla^s u||_{L^2(V)} \right ),
	\end{align*} 
	so that for $C_3=C_2 k$ we have
	\begin{equation} \label{loworderterm1}
	||\nabla^s u||_{L^p(U_\star)} \leq C_3 \left (||f ||_{L^p(V)} +  ||\nabla^s {g}||_{L^p(V)} + ||\nabla^s u||_{L^2(V)} \right ). 
	\end{equation}
	Next, consider the general case when $f \in L^{p_\star}_{loc}(\Omega)$, where $p_\star = \max \left \{\frac{pn}{n+ps},2 \right \}$. Consider the function $f_{V}:\mathbb{R}^n \to \mathbb{R}$ defined by
	$$ f_{\Omega}(x) := \begin{cases} 
	f(x), & \text{if } x \in V \\
	0, & \text{if } x \in \mathbb{R}^n \setminus V 
	\end{cases}$$ 
	and note that $f_{V} \in L^{p_\star}(\mathbb{R}^n) \cap L^2(\mathbb{R}^n)$. 
	By \cite[Proposition 5.1]{Me}, there exists a unique weak solution $h \in H^{s,2}(\mathbb{R}^n) \subset H^{s,2}_{loc}(\Omega | \mathbb{R}^n)$ of the equation
	\begin{equation} \label{helpeq}
	(-\Delta)^s h + h = f_V \quad \text{in } \mathbb{R}^n,
	\end{equation}
	where
	$$ (-\Delta)^s h(x) = C_{n,s} \int_{\mathbb{R}^n} \frac{h(x)-h(y)}{|x-y|^{n+2s}}dy$$
	is the fractional Laplacian of $h$.
	In view of Proposition \ref{altcharBessel}, Theorem \ref{BesselEmbedding}, using the classical $H^{2s,p_\star}$ estimates for the fractional Laplacian on the whole space $\mathbb{R}^n$ (cf. for example \cite[Lemma 3.5]{KassMengScott}) and setting $\widetilde h(x,y) := C_{n,s}(h(x)-h(y))$, we obtain
	\begin{equation} \label{gf}
	\begin{aligned}
	||h||_{L^p(\mathbb{R}^n)} + ||\nabla^s \widetilde h||_{L^p(\mathbb{R}^n)} = & C_4 (||h||_{L^p(\mathbb{R}^n)} + ||\nabla^s h||_{L^p(\mathbb{R}^n)}) \\ \leq & C_5 ||h||_{H^{s,p}(\mathbb{R}^n)} \\ \leq & C_6 ||h||_{H^{2s,p_\star}(\mathbb{R}^n)} \\ \leq & C_7 ||f_V||_{L^{p_\star}(\mathbb{R}^n)} = C_6 ||f||_{L^{p_\star}(V)},
	\end{aligned}
	\end{equation}
	where all constants depend only on $n,s$ and $p$.
	Furthemore, $u$ is a weak solution of  
	$$ L_A^\Phi u = L \widetilde g + h \text{ in } V,$$
	where $$\widetilde g(x,y):= g(x,y)+\widetilde h(x,y).$$ Therefore, by combining the estimates (\ref{loworderterm1}) and (\ref{gf}), we arrive at
	\begin{align*}
	||\nabla^s u||_{L^p(U_\star)} \leq & C_3 \left (||h||_{L^p(V)} + ||\nabla^s \widetilde g||_{L^{p}(V)} +  ||\nabla^s u||_{L^2(V)} \right ) \\
	\leq & C_3 \left (||h ||_{L^{p}(V)} +  ||\nabla^s \widetilde h||_{L^p(V)} +  ||\nabla^s {g}||_{L^p(V)} + ||\nabla^s u||_{L^2(V)} \right ) \\
	\leq & C \left (||f ||_{L^{p_\star}(V)} +  ||\nabla^s {g}||_{L^p(V)} + ||\nabla^s u||_{L^2(V)} \right )
	\end{align*}
	for some constant $C=C(p,n,s,\lambda)>0$, which proves the estimate (\ref{Sobest}). In particular, since by assumption we have $f,\nabla^s g \in L^p(V)$, we obtain that $\nabla^s u \in L^p(U_\star)$.
	Let us now prove that $u \in H^{s,p}(U_\star | \mathbb{R}^n)$. 
	For any $r \in [2,p]$, define
	$$ r^\star := \begin{cases} 
	\min \{\frac{rn}{n-rs},p\}, & \text{if } rs < n \\
	p, & \text{if } rs \geq n .
	\end{cases}$$ 
	By Proposition \ref{altcharBessel} and Theorem \ref{BesselEmbedding}, for any $r \geq 2$ we have
	$$ H^{s,r}(U_\star | \mathbb{R}^n) \subset H^{s,r}(U_\star) \hookrightarrow L^{r^\star}(U_\star).$$
	Since $u \in H^{s,2}(U_\star|\mathbb{R}^n)$, we have $u \in L^{2^\star}(U_\star)$.
	If $p = 2^\star$, we have $u,\nabla^s u \in L^{p}(U_\star)$ and therefore $u \in H^{s,p}(U_\star | \mathbb{R}^n).$ If $p> 2^\star$, then we have $u,\nabla^s u \in L^{2^\star}(U_\star)$, so that $u \in H^{s,2^\star}(U_\star | \mathbb{R}^n).$ We therefore arrive at $u \in L^{{2^\star}^\star}(U_\star)$. If ${2^\star}^\star = p$, then we have $u,\nabla^s u \in L^{p}(U_\star)$ and therefore $u \in H^{s,p}(U_\star | \mathbb{R}^n).$ If ${2^\star}^\star > p$, then iterating the above procedure also yields $u \in H^{s,p}(U_\star | \mathbb{R}^n)$ and therefore $u \in H^{s,p}(U | \mathbb{R}^n)$ at some point. Since $U$ is an arbitrary relatively compact open subset of $\Omega$, we conclude that $u \in H^{s,p}_{loc}(\Omega | \mathbb{R}^n)$. 
	This finishes the proof.
\end{proof}

\bibliographystyle{amsalpha}

\end{document}